\documentclass[12pt]{amsart}

\numberwithin{equation}{section}

\usepackage{amsmath,amsthm,amsfonts,eucal,}
\usepackage{xypic}
\usepackage{graphicx}
\usepackage{amssymb}
\usepackage{mathtools}
\usepackage{url}

\def\cb{{\mathcal B}}
\def\cc{{\mathcal C}}

\def\ch{{\mathcal H}}

\def\cs{{\mathcal S}}

\def\ga{{\mathfrak A}} 
\def\gb{{\mathfrak B}}

\def\bb{{\mathbb B}}
\def\bc{{\mathbb C}}

\def\bj{{\mathbb J}}

\def\bn{{\mathbb N}}

\def\bp{{\mathbb P}}
\def\br{{\mathbb R}}
\def\bt{{\mathbb T}}

\def\bz{{\mathbb Z}}

\def\a{\alpha}

\def\g{\gamma}

\def\p{\pi}

\def\s{\sigma} 

\def\f{\varphi}  
  
\def\om{\omega}

\newtheorem{thm}{Theorem}[section]
\newtheorem{example}[thm]{Example}
\newtheorem{lem}[thm]{Lemma}
\newtheorem{cor}[thm]{Corollary}
\newtheorem{prop}[thm]{Proposition}

\theoremstyle{definition}
\newtheorem{rem}[thm]{Remark}
\newtheorem{defin}[thm]{Definition}

\begin{document}
\title[On the Ryll-Nardzewski Theorem]
{On the Ryll-Nardzewski Theorem for Quantum Stochastic Processes}

\author{Valeriano Aiello}
\address{ Valeriano Aiello\\
Dipartimento di Matematica\\
Universit\`{a} Sapienza di Roma\\
 Roma, Italy}
\email{\texttt{valerianoaiello@gmail.com}}

\author{Simone Del Vecchio}
\address{Simone Del Vecchio\\
Dipartimento di Matematica\\
Universit\`{a} degli studi di Bari\\
Via E. Orabona, 4, 70125 Bari, Italy}
\email{\texttt{simone.delvecchio@uniba.it}}

\author{Stefano Rossi}
\address{Stefano Rossi\\
Dipartimento di Matematica\\
Universit\`{a} degli studi di Bari\\
Via E. Orabona, 4, 70125 Bari, Italy}
\email{\texttt{stefano.rossi@uniba.it}}

\begin{abstract}
We prove a Ryll-Nardzewski Theorem for quantum stochastic processes, that shows that under natural assumptions which generalize the classical probability setting, the distributional symmetries of exchangeability and spreadability are the same. 
We further show that product states on twisted tensor products of $C^*$-algebras provide a source of counterexamples to the Ryll-Nardzewski theorem, namely of quantum stochastic processes which are spreadable but not exchangeable.
Furthermore, in this setting, we also analyze braidability of product states.
We then prove an extended de Finetti Theorem for quantum stochastic processes whose distribution factorizes through twisted tensor products.

\end{abstract}

\maketitle

\section{Introduction}

Exchangeability is a remarkable distributional symmetry that a sequence of random variables enjoys when its joint 
distribution is invariant under all permutations. Basic examples of such sequences are obtained by taking sequences of i.i.d.  (independent and identically distributed) random variables. De Finetti's celebrated theorem says that, in general, the distribution
of an exchangeable sequence is a mixture ({\it i.e.} a convex combination) of i.i.d. variables. The theorem can also be recast in terms of conditional independence: a sequence is exchangeable if and only if it is conditionally independent with respect to the tail algebra, the intuition beneath this statement being that the variables would  become independent if the value taken by an unaccessible random variable (the one that generates the tail algebra) were known.
The picture is then completed by the Ryll-Nardzewski theorem that, for a sequence of random variables, exchangeability is  the same as spreadability, the seemingly weaker symmetry where the joint distribution is invariant under strictly increasing maps. Phrased differently, spreadability only requires invariance of the joint distribution under taking subsequences.\\
The more general non-commutative setting is much less clear-cut.
While in the framework of non-commutative von Neumann algebras it is still possible to obtain a version
of de Finetti's theorem formulated in terms of conditional independence, which is done by K\"{ostler}
in \cite{K}, spreadability is  a strictly weaker distributional symmetry than exchangeability instead, as shown in
\cite{K}. Nevertheless, non-commutative models do exist where spreadability and exchangeability coincide.
An important case in point is given by the states on the CAR algebra, for which invariance under permutations or under increasing maps is the same, as recently proved in \cite{CDRCAR}. Exchangeable states on the CAR algebra had already been addressed in \cite{CFCMP}, where they were shown to make up a Bauer simplex whose boundary is given by all infinite products of a single ($\bz_2$-invariant) state on $\mathbb{M}_2(\bc)$. As remarked in
\cite{CDRCAR}, it is not the CAR relations that force the Ryll-Nardzewski to hold but rather the ($\bz_2$-graded) tensor structure of the CAR algebra. Indeed, the Ryll-Nardzewski theorem will hold true in any infinite tensor product, and, more in general, in any $\bz_2$-graded tensor product of an assigned unital $C^*$-algebra. Furthermore, exchangeable states on tensor products are entirely known since they were completely classified in the influential work of St\o rmer \cite{Sto},
where it was shown they make up a Bauer simplex with boundary given by all infinite products of a single state. Note that the classical case corresponds to starting with a commutative $C^*$-algebra.\\
Interestingly, all the specific non-commutative models we mentioned above are quotients of the infinite free product
of a given $C^*$-algebra, which we refer to as the sample algebra, with itself countably many times.
A natural question we  answer in this paper is to determine to what extent the Ryll-Nardzewski
theorem continues to hold on more general quotient than tensor products. To do so, we first establish our setting
by considering quantum stochastic processes whose distribution, which, as we will recall at length, is a state on the infinite free product of the sample $C^*$-algebra,  factorizes through a quotient of the free product.
The main result we obtain in this direction is Theorem \ref{mainRyll}, where we prove that exchangeability follows from spreadability provided that the quotient $\ast_\bn\ga/I$ is by an ideal $I$ that, first, is invariant
under the natural action (on indices) of permutations on the free product $\ast_\bn\ga$, and, second, dials down
the too high degree of non-commutativity of the free product, in the sense that the quotient is linearly
generated, as a Banach space, by {\it ordered} monomials, {\it i.e.} elements of the type
$p_I((i_{j_1}(a_1)i_{j_2}(a_2)\cdots i_{j_n}(a_n))$ with $j_1<j_2<\ldots< j_n$ (or the other way round), where $p_I$ is the canonical
projection onto the quotient and the $i_j$'s are the embeddings of $\ga$ into $\ast_\bn\ga$.
It is worth noting that our non-commutative version of the Ryll-Nardzewski
theorem  is general enough to apply to the above mentioned models as well as twisted graded tensor products.\\
Twisted tensor products of $G$-graded $C^*$-algebras, with $G$ a compact abelian group,
are relevant instances of such quotients which generalize usual tensor products by introducing
commutation rules between tensors localized in disjoint sites depending, through a bicharacter $v$ of the dual group $\widehat{G}$, on the
grade of the tensors. The dependence is such that two tensors $a, b$ localized in disjoint sites commute if and only if the bicharacter $v$ is trivial on the corresponding grades, namely if $v(\partial a, \partial b)=1$.\\
A stark difference with usual tensor products, though, is that the commutation rules alluded to  may and in general will forbid permutations from acting naturally on twisted products, unless the bicharacter is antisymmetric. Even so,  displacement of indices (those keeping track of the localization of tensors) by increasing maps is always compatible with the commutation rules, which makes it possible to consider spreadability in all cases.\\
Now, a remarkable construction that naturally comes along with tensor products is of course that of product states.
This construction carries over to the more general context of twisted products as long as more care is taken.
In fact, it was already known, see {\it e.g.} \cite{FidVic},  that the  infinite product state of a single given state $\om$ on a $G$-graded
$C^*$-algebra $\ga$ exists as a state, namely as a normalized {\it positive} linear functional, on the infinite twisted tensor product $\otimes_v^\bn\ga$, if the state $\om$ is $G$-invariant. However, what we show in the present paper is that
it is possible to go well beyond $G$-invariance.  Precisely, in Theorem \ref{existenceprod} we prove that a necessary and sufficient condition for the product state to exist is that $\om$ is compatible with the grading in the following sense: $v(\eta, \chi)=1$ if $\eta, \chi\in\widehat{G}$ sits in what we call the spectral support, ${\rm supp}_G(\om)$, of the state $\om$. This is by definition the set of characters $\chi$ in $\widehat{G}$ for which there exists
$a$ in $\ga$ with $\partial a=\chi$ and $\om(a)\neq 0$.\\
One of the reasons that much importance is attached to product states in the paper is  they spawn the
possibly simplest examples,  at the level of free products, of states that fail to be exchangeable while being spreadable, Proposition \ref{counterexamples}. Another is they yield an example of a process that is spreadable but not  braidable in a natural sense in our framework, Theorem \ref{counterexamplebraid}. This can be seen as a first step towards  the answer to the question as to whether there exists a spreadable process which is not braidable
raised by Evans, Gohm, and Köstler in \cite{EGK17}, after the last two authors, working in the framework of von Neumann algebras,
introduced in \cite{GK08}  a new  symmetry, braidability, which is intermediate between exchangeability and spreadability.\\ 
We then move on to study spreadability on twisted tensor products. Since the action of monotone maps is only implemented by
(injective) endomorphisms, in order to make non-commutative ergodic theory applicable, we employ a thickening (of the indices labelling the infinite product)
procedure first introduced in \cite{CDRCAR} and resorted to in \cite{CDGR}, which allows us to get an automorphic action on a larger $C^*$-algebra of a group $H$ containing the original semigroup of increasing maps. 
The $H$-abelianness of the enlarged system is  analyzed
thoroughly in Theorem \ref{suffcondabelianness}, which provides a useful, general sufficient condition, and Theorem \ref{charHabelianness}, which gives a necessary and sufficient condition. Working under the sufficient condition, 
in Theorem \ref{maindefinetti} we give our version of de Finetti's theorem on infinite twisted tensor products showing that
the compact convex set of all spreadable states on such chains is the Bauer simplex whose boundary is exactly the compact set of all product states. Corollary \ref{definettiprocesses} states the corresponding result for quantum stochastic processes whose distribution factorizes through any of our infinite twisted tensor products.
Interestingly, these hypotheses are satisfied in the models of parafermion algebras considered in
\cite{BJLW}. These algebras are acted upon by the braid group in a natural way.
In Proposition \ref{finprod} we show that the sufficient condition for $H$-abelianness to hold
covers a variety of situations, such as symmetric bicharacters on finite direct product groups of the type
 $\bz_{n_1}\times....\times\bz_{n_k}$, with $n_1, \ldots, n_k$ being odd. These lead to what we deem is a first example of a non-commutative model where not all spreadable states can be obtained out of a single tensor product subalgebra by composition with the canonical conditional expectation onto it, Example \ref{definettifoliated}.\\
Finally, stationary states on infinite twisted products are addressed as well.  In Theorem \ref{sufficientPoulsen}
we give a sufficient condition for the convex compact set of stationary states to be (affinely isomorphic with)
the Poulsen simplex, which is also necessary when the given grading on the sample algebra is ergodic.\\
As an outlook for future developments of the present work, we would like to point out that the thickening procedure provides a way to associate a unitary representation of the Thompson group to any spreadable quantum stochastic process (see Remark \ref{representations}).
The representations thus obtained might be worth comparing with those constructed in \cite{KK, KKW}.\\
In an effort to keep the exposition as self-contained as possible, we have gleaned in the preliminary section the basic definitions and results concerning quantum stochastic processes, the ergodic theory of $C^*$-dynamical systems, and a quick account on twisted tensor products of $G$-graded $C^*$-algebras.

\section{Preliminaries}
\subsection{Quantum stochastic processes}
In this paper, $\bn$ is the set of all natural numbers, including $0$. Quantum probability is here considered in its $C^*$-algebraic setting, thus a discrete  quantum stochastic process  will be
 a quadruple
$\big(\ga,\ch,\{\iota_j\}_{j\in \bn},\xi\big)$, where $\ga$ is a (unital) $C^{*}$-algebra, referred to as the sample
algebra of the process, $\ch$ is a Hilbert space,
whose inner product is denoted by $\langle\cdot,\cdot \rangle$, and is taken linear to the left and anti-linear to the right,
the maps $\iota_j$ are (unital) $*$-representation of $\ga$ on $\ch$, and
$\xi\in\ch$ is a unit vector, which is cyclic for  the von Neumann algebra
$\bigvee_{j\in \bn}\iota_j(\ga)$.\\
The above definition of a quantum stochastic process is  as general as to include
all classical cases as realized in the Kolmogorov consistency theorem. Indeed, suppose one is given a  family of finite-dimensional 
 distributions $\mu_{j_1, j_2, \ldots, j_k}\in P(\br^k)$, for any $k\in\bn$ and indices
$j_1, j_2, ..., j_k\in \bn$.
Consider the Tychonoff product $\br^\bn$, whose points are denoted by
$x=(x_j)_{j\in \bn}$. For any $k, j_1, j_2, \ldots, j_k\in \bn$, and $A_1, A_2, \ldots, A_k$ Borel subsets of $\br$, we define
a corresponding cylinder of $\br^\bn$ as
$$C_{j_1, j_2, \ldots, j_k}^{A_1, A_2, \ldots, A_k}:=\big\{x\in \br^\bn \mid x_{j_1}\in A_1, x_{j_2}\in A_2, \ldots, x_{j_k}\in A_k\big\}.$$
As is known, if the finite-dimensional distributions satisfy the so-called consistency conditions of
Kolmogorov's theorem, then a measure $\mu$ can be defined on the $\sigma$-algebra
$\mathfrak{C}$
generated by cylinders by 
\begin{equation*}\label{cylmeasure}
\mu\big(C_{j_1, j_2, \ldots, j_k}^{A_1, A_2, \ldots, A_k}\big)= \mu_{j_1, j_2, \ldots, j_k}
(A_1\times A_2\times\ldots\times A_k).
\end{equation*}
Denote by $X_j$ the $j$-th coordinate function on $\br^\bn$, namely
$X_j(x)=x_j$, $x\in \br^\bn$.
The functions $\{X_j\}_{j\in \bn}$ provide a stochastic process on the probability space
$(\br^\bn, \mathfrak{C}, \mu)$ having $\{\mu_{j_1, j_2, \ldots, j_k}\}$ as finite-dimensional distributions.\\
The above construction can also be reinterpreted in terms of representations of a given sample algebra.
To this end, take $\ga=C_0(\br)$.
With $\ch:=L^2(\br^\bn, \mu)$, we consider a  family of $*$-homomorphisms of $\ga$ 
to $\mathcal{B}(\ch)$ defined as
$$
\iota_j(f):= f(X_j), \,  \quad f\in C_0(\br)\,,
$$
where $f(X_j)$ is the bounded operator acting on $L^2(\br^\bn, \mu)$  by multiplication
by the function $f(X_j)$.\\
Denoting by $1\in L^2(\br^\bn, \mu)$ the function  equal to 
$1$ $\mu$-a.e., we obviously have that $(C_0(\br), L^2(\br^\bn, \mu), \{\iota_j\}_{j\in \bn}, 1)$ is a stochastic process and
corresponds to the realization of the process provided by  Kolmogorov's theorem
through the coordinate functions $X_j$.\\
Conversely, a process is classical exactly when both the sample algebra
$\ga$ and $\bigvee_{j\in \bn}\iota_j(\ga)$ are commutative, in that one has the following, see {\it e.g.} the appendix of \cite{CR}.
\begin{prop}
Let $(C_0(\br), \ch, \{\iota_j\}_{j\in\bn}, \xi)$ be
a stochastic process such that the von Neumann algebra $\bigvee_{j\in\bn}\iota_j (C_0(\br))$ is commutative.
Then there exists a commuting family $\{A_j: j\in\bn\}$ of
(possibly unbounded) self-adjoint operators on $\ch$ such that for every $j\in\bn$ one has
$\iota_j(f)=f(A_j)$, $f\in C_0(\br)$.\\
In addition there exist a probability measure $\mu$ on $(\br^\bn, \mathfrak{C})$ and a unitary $U: \ch\rightarrow L^2(\br^\bn, \mu)$ such that:
\begin{enumerate}
\item $U\xi=[1]_\mu$
\item$UA_jU^*=M_j$, for every $j\in\bn$, where $M_j$ is the operator acting on $L^2(\br^\bn, \mu)$
as the multiplication by $X_j$.
\end{enumerate}
\end{prop}

When the sample algebra $\ga$ is unital, as will be in this paper, 
the assignment of a quadruple $\big(\ga,\ch,\{\iota_j\}_{j\in \bn },\xi\big)$ amounts to a state $\varphi$ on
the free product $C^*$-algebra $\ast_{\bn} \ga$, which is the distribution of the whole process.
We recall that  $\ast_{\bn} \ga$ is the $C^*$-algebra determined by the following  universal property: there is
a family $\{i_j: j\in \bn \}$ of $*$-homomorphisms $i_j:\ga\rightarrow\ast_{\bn} \ga$ such that if
$\cb$ is a second $C^*$-algebra with a family of  $*$-homomorphisms $\{i_j': j\in \bn \}$ from $\ga$ to $\cb$, then there exists a unique $*$-epimorphism $\Phi:\ast_{\bn} \ga\rightarrow\cb$ with
$i_j'=i_j\circ\Phi$. Informally, $\ast_{\bn} \ga$ is the maximal completion of the $*$-algebra linearly spanned by all words of the form $a_{j_1}a_{j_2}\cdots a_{j_n}$ with $n$ in $\bn$, $j_1, \ldots, j_n\in \bn $,
$j_1\neq j_2\neq\ldots\neq j_n$, and $a_{j_1}, a_{j_2}, \ldots, a_{j_n}$ in $\ga$.
The universal property implies in particular that any function $f: \bn\rightarrow \bn$ can be lifted to a
$*$-endomorphism $\Phi_f$ of  $\ast_{\bn} \ga$ uniquely determined by $\Phi_f(i_j(a))=i_{f(j)}(a)$, for all
$j$ in $\bn$ and $a$ in $\ga$.\\
For more on free products of $C^*$-algebras the reader is referred to e.g.
\cite{VDN92}. For completeness, we quickly recall how the correspondence between stochastic
processes and states on the free product $C^*$-algebra works. If
one starts
with a stochastic process, then a state
$\varphi$ on the free product $\ast_{\bn} \ga$ can be defined by setting
\begin{equation}\label{corresp}
\varphi(i_{j_1}(a_1)i_{j_2}(a_2)\cdots i_{j_n}(a_n)):=\langle(\iota_{j_1}(a_1)\iota_{j_2}(a_2)\cdots\iota_{j_n}(a_n)\xi, \xi\rangle
\end{equation}
for all $n\in\bn$, $j_1\neq j_2\neq\cdots\neq j_n\in \bn $ and
$a_1, a_2, \ldots, a_n\in\ga$, where $i_j:\ga\rightarrow \ast_{\bn}\ga$ is
the $j$-th embedding of $\ga$ into $\ast_{\bn} \ga$.
The values of $\varphi$ on monomials of the type above are
often referred to as finite-dimensional distributions of the process itself.\\
Conversely, all states on the free product $\ast_{\bn} \ga$ are easily seen to arise in this way.
Phrased differently, starting now with a state $\f\in\cs\big(\ast_{\bn}\ga\big)$,  a stochastic process can be defined by using  the GNS representation $(\p_\f, \ch_\f, \xi_\f)$ of
$\f$. Indeed, for every $j\in \bn $ we can set $\iota_j(a):=\pi_\f(i_j(a))$, $a\in\ga$, so as to get the quadruple
$\big(\ga,\ch_\f,\{\iota_j\}_{j\in \bn },\xi_\f\big)$.

\subsection{Spreadability and exchangeability}
We denote by $\bp_\bn$ the group of  bijective maps of $\bn$ which only move finitely many naturals.  The group operation
is given by the map composition.\\
A quantum stochastic process $\big(\ga,\ch,\{\iota_j\}_{j\in \bn },\xi\big)$ is said exchangeable
if its finite-dimensional distributions are invariant under permutations, namely if for all $\sigma$ in $\bp_\bn$ one has
$$\langle(\iota_{\s(j_1)}(a_1)\iota_{\s(j_2)}(a_2)\cdots\iota_{\s(j_n)}(a_n)\xi, \xi\rangle=\langle(\iota_{j_1}(a_1)\iota_{j_2}(a_2)\cdots\iota_{j_n}(a_n)\xi, \xi\rangle\,,$$
for all $j_1\neq j_2\neq\ldots\neq j_n$ and $a_1, a_2,\ldots, a_n$ in $\ga$.\\
This invariance property may of course be phrased in terms of the corresponding state $\varphi$
on $\ast_\bn\ga$.
Indeed, if $\a: \bp_\bn\rightarrow{\rm Aut}(\ast_{\bn} \ga)$ is the natural action of $\bp_\bn$ 
on $\ast_{\bn} \ga$, then the process is exchangeable if and only if $\varphi\circ\a_\s=\varphi$, for all
$\s$ in $\bp_\bn$.\\
Classically, a sequence of random variables is spreadable if all of its subsequences have the same joint
distribution. Thus, a quantum stochastic process is spreadable if, for every monotone $f:\bn\rightarrow\bn$, one has
$$\langle(\iota_{f(j_1)}(a_1)\iota_{f(j_2)}(a_2)\cdots\iota_{f(j_n)}(a_n)\xi, \xi\rangle=\langle(\iota_{j_1}(a_1)\iota_{j_2}(a_2)\cdots\iota_{j_n}(a_n)\xi, \xi\rangle\,,$$
for all $j_1\neq j_2\neq\ldots\neq j_n$ and $a_1, a_2,\ldots, a_n$ in $\ga$.\\
Spreadability is a symmetry which can be implemented by the natural action 
of the unital semigroup of all strictly increasing maps of $\bn$ to itself by endomorphisms of the free product
$\ast_\bn\ga$. Since distributional symmetries
only involve finitely many indices at a time, a process is spreadable if and only if its distribution is invariant under the action of the smaller semigroup
$\bj_\bn$ made up of all
maps $f$ of $\bn$ to itself whose range is cofinite, that is $|\bn\setminus f(\bn)|<\infty$, where for any
set $E$, $|E|$ denotes the cardinality of $E$.
We recall that $\bj_\bn$ is countably generated. Precisely, for every fixed $h\in\bn$, define the $h$-{right hand-side partial shift} as the element $\theta_h$ of $\bj_\bn$ given by
$$
\theta_h(k):=\left\{\begin{array}{ll}
                      k & \text{if}\,\, k<h\,, \\
                      k+1 & \text{if}\,\, k\geq h\,.
                    \end{array}
                    \right.
$$
It is easy to see that  the set $\{\theta_h : h\in\bn\}$ generates
$\bj_\bn$ as a unital semigroup.\\
Analogously, $\bj_\bz$ is the semigroup of increasing maps of $\bz$ into itself with cofinite range.
 $\bj_\bz$ is finitely generated as well, with a set of generators being $\{\tau, \tau^{-1}, \theta_0\}$, with
$\tau$ being given by $\tau(k)=k+1$, $k$ in $\bz$, and
$\theta_0$ being understood as an increasing map of $\bz$ to itself defined by the same formula as above. 
\begin{rem}
In this paper we have adopted  $\bn$ as the index set for  our quantum stochastic processes.
However, all the results hold with $\bz$ as well. 
\end{rem}

\subsection{$C^*$-dynamical systems} A $C^*$-dynamical system is a triple $(\ga, H, \a)$, where $\ga$ is a (unital)
$C^*$-algebra, $H$ a locally compact ({\it e.g.} discrete) group, and
$\a: H\rightarrow{\rm Aut}(\ga)$ a group homomorphism, that is
$\a_{hk}=\a_h\circ\a_k$, for all $h, k$ in $H$.\\
Given a $C^*$-dynamical system $(\ga, H, \a)$, we denote by $\cs^H(\ga)$ the compact convex set of its invariant states, that is
 $$\cs^H(\ga)
 :=
 \{\varphi\in\cs(\ga): \varphi\circ\a_h=\varphi,\textrm{for all}\,\, h\in H\}\,.$$
Extreme states, namely extreme points of the compact convex set $\cs^H(\ga)$ are often referred to as the ergodic
states of the system, and denoted by ${\rm Extr}(\cs^H(\ga))$.\\
On the GNS representation $(\pi_\varphi, \ch_\varphi, \xi_\varphi)$ of any invariant state $\varphi$, the action of the group is unitarily implemented. Precisely, 
$$u_h^\varphi\pi_\varphi(a)\xi_\varphi:=\pi_\varphi(\a_h(a))\xi_\varphi\,, \,\,a\in\ga\,,$$
defines a unitary operator  on $\ch_\varphi$ such that $u_h^\varphi \pi_\varphi(a) (u_h^\varphi)^*=\pi_\varphi(\a_h(a))$, $a$ in $\ga$. \\
Denote by $\ch_\varphi^H$ the closed subspace of all
invariant vectors for the group of unitaries $\{u_h^\varphi: h\in H\}$, that is
$\ch_\varphi^H := \{\xi\in\ch_\varphi: u_h^\varphi\xi=\xi, h\in H\}$.  Note that
$\bc\xi_\varphi\subset\ch_\varphi^H$.
The orthogonal
projection of $\ch_\varphi$ onto $\ch_\varphi^H$ will be denoted by $E_\varphi$.\\
To keep the exposition of the present paper as self-contained as possible, we collect some known general results on
$C^*$-dynamical systems below. Before doing this, we also recall what $H$-abelianness of a system is.
A $C^*$-dynamical system is $H$-abelian if, for every $\varphi$ in $\cs^H(\ga)$, the  subspace
$E_\varphi\pi_\varphi(\ga)E_\varphi\subset \cb(\ch_\varphi)$ is commutative.
\begin{lem}\label{summary}
Given a $C^*$-dynamical system $(\ga, H, \a)$, and $\varphi$ in $\cs^H(\ga)$, consider the following
\begin{itemize}
\item [(i)] $\varphi$ is ergodic;
\item [(ii)] $\{\pi_\varphi(\ga), 	\{u_h^\varphi: h\in H\}\}'=\bc 1_{\ch_\varphi}$\;
\item [(iii)] $E_\varphi$ is the orthogonal projection onto $\bc\xi_\varphi$.
\end{itemize}
Then $(i)$ and $(ii)$ are equivalent, and $(iii)$ implies both. In addition, the system is
$H$-abelian if and only if the three conditions are all equivalent for all $\varphi$ in $\cs^H(\ga)$.
Finally, $\cs^H(\ga)$ is a Choquet simplex if and only if
the system is $H$-abelian.
\end{lem}
For the proof, see \cite[Propositions 3.1.10, 3.1.12, 3.1.14]{S} and \cite[Corollary 4.4]{B}.

\subsection {Twisted tensor products of $G$-graded $C^*$-algebras}
Infinite twisted tensor products of graded $C^*$-algebras are the non-commutative models we will be chiefly
concerned with in this paper. To keep the exposition as self-contained as possible, we 
briefly review what twisted tensor  products are.\\
Henceforth, $G$ will always denote a compact abelian group, with (normalized) Haar measure 
$\mu_G$. Its Pontryagin dual will be denoted by $\widehat{G}$.
A $G$-graded $C^*$-algebra is a $C^*$-dynamical system
$(\ga, G, \gamma)$, where $\ga$ is a unital $C^*$-algebra, and
$\gamma: G\rightarrow{\rm Aut}(\ga)$ a group homomorphism, pointwise norm continuous, {\it i.e.}, for
every $a$ in $\ga$, the function $G\ni g\mapsto \gamma_g(a)\in\ga$ is continuous w.r.t the norm topology.
The eigenspace of the action associated with  $\chi$ in $\widehat{G}$ is the (possibly zero) subspace of
$\ga$ given by
$$V_\chi :=\{a\in\ga: \gamma_g(a)=\chi(g)a, \textrm{for all}\, g\in G\}\,.$$
For every $\chi, \eta$ in $\widehat{G}$, one has
$V_\chi V_\eta\subset V_{\chi\eta}$ and $V_{\chi}^*=V_{\chi^{-1}}$.
As is well known, eigenspaces corresponding to different characters are linearly independent and
the algebraic direct sum $\oplus_{\chi\in\widehat{G}} V_\chi$ is a dense $*$-algebra of $\ga$, see {\it e.g.}
\cite[Theorem 2.5]{Salle16}, known as the
algebraic layer of $\ga$, which we denote by $\ga_0$. This is quite a rigid object in so far as  the
$C^*$-maximal seminorm of $\ga_0$ coincides with the given  norm of $\ga$, \cite[Theorem 4.1]{FidVic}.\\
An element $a$ in $\ga$ is said to be homogeneous if it belongs to  $V_\chi$ for some $\chi$ in $\widehat{G}$, which is the degree of $a$ and is denoted by $\partial a$.\\
In general, the subspaces $V_\chi$ may well be zero for some characters $\chi$. However, if the action is faithful, that is $G\ni g\mapsto\gamma_g\in {\rm Aut}(\ga)$ is an injective homomorphism, and topologically ergodic, that is
$\ga^G :=\{a\in\ga: \gamma_g(a)=a, \textrm{for all}\, g\in G\}=\mathbb{C}$, then none of the $V_\chi$'s can be zero. This should be a known fact. Nevertheless, we include a proof for want of a reference.
\begin{prop}\label{fullspectrum}
If the $C^*$-dynamical systems $(\ga, G, \gamma)$ is faithful and topologically ergodic, then for any
$\chi$ in $\widehat{G}$ one has $V_\chi\neq \{0\}$.
\end{prop}
\begin{proof}
Set $\sigma(\gamma) := \{\chi\in \widehat{G}: V_\chi\neq \{0\}\}$.
As is known, topological  ergodicity  implies that
$\sigma(\gamma)$ is a subgroup of $\widehat{G}$.\\
 For completeness' sake, we also recall how this can be
seen. For any $\chi$ in $\sigma(\gamma)$, $V_\chi$ is a one-dimensional
subspace generated by a unitary: if $a$ is in $V_\chi$, then $a^*a$ lies in $\ga^G$, which means
$a^*a=\lambda$, for some scalar $\lambda>0$, hence $a$ is a multiple of a unitary; second, if $u, v$ are unitaries $V_\chi$, then $u^*v=z$, for some phase $z$ in $\bt$, that is $v=zu$.
Now, for $\chi, \eta$ in $\sigma(\gamma)$, the spectral  subspace $V_{\chi\eta}$ is  still non-zero as it contains $V_\chi V_\eta$, which is certainly non-zero because the product of unitaries is clearly non-zero being unitary.\\
We are now ready to show that $\sigma(\gamma)$ is the whole dual group $\widehat{G}$. Suppose, on the contrary,
that $\sigma(\gamma)\subsetneq\widehat{G}$.  Since $\sigma(\gamma)$ is a proper (closed) subgroup of
$\widehat{G}$, there must exist $g_0\neq e$ in $G$ such that $\chi(g_0)=1$ for all $\chi$ in $\sigma(\gamma)$.
From the spectral decomposition of $\ga$, we immediately see that
$\gamma_{g_0}={\rm id}_\ga$, hence $\a$ is not an injective homomorphism.
\end{proof}
If $(\ga, G, \gamma)$ is a $G$-graded $C^*$-algebra, we denote
by $\cs^G(\ga)$ the (compact convex) set of $G$-invariant states of $\ga$, that is
$$\cs^G(\ga) :=\{\om\in\cs(\ga): \om\circ\gamma_g=\om, \textrm{for all}\, g\in G\}\,.$$
We recall that a bicharacter on $\widehat{G}$ is a (continuous) function
$v: \widehat{G}\times\widehat{G}\rightarrow \bt$ such that,
for any fixed $\chi$ in $\widehat{G}$, the functions $\widehat{G} \ni \eta\mapsto v(\eta, \chi)\in\bt$ and
$\widehat{G} \ni \eta\mapsto v(\chi, \eta)\in\bt$  are characters of $\widehat{G}$ (and thus elements of $G$ via  Pontryagin duality).
A bicharacter $v$ is said to be symmetric if $v(\chi, \eta)=v(\eta, \chi)$, antisymmetric if
$v(\chi, \eta)= \overline{v(\eta, \chi)}$, for all
$\chi, \eta $ in $\hat{G}$.
If $v$ is a bicharacter on $\widehat{G}$, $v_S(\chi, \eta) := v(\chi, \eta)v(\eta, \chi)$, $\chi, \eta$ in $\widehat{G}$, is still
a bicharacter, known as the symmetrized of $v$.\\
Associated with any bicharacter $v$ there is a subset $\Delta_v\subset\widehat{G}$, defined as
$$\Delta_v := \{\chi\in\widehat{G}: v(\chi, \chi)=1\}\,.$$
Note that $\Delta_v$ is a subgroup if and only if $v_S$ is $1$ on $\Delta_v\times\Delta_v$. In particular,
$\Delta_v$ is a subgroup if $v$ is antisymmetric.
By a slight abuse of terminology, $\Delta_v$ is often referred to as the isotropy subgroup even if
it fails to be a subgroup.\\
Given a $G$-graded $C^*$-algebra $\ga$ and a bicharacter $v$,  we denote by
$\bigotimes_v^\bn\ga$ the infinite twisted (w.r.t. $v$) maximal tensor product of $\ga$ 
countably many  times  with itself.
This is a $C^*$-algebra $\cb$ with a family $\{i_k: k\in\bn\}$ of (injective)
$*$-homorphisms $i_k: \ga\rightarrow \cb$ such that
for every homogeneous $a, b$ in $\ga$, and for every $k<l \in\bn$, one has
\begin{equation}\label{commrules}
i_k(a)i_l(b)=v(\partial a, \partial b)i_l(b)i_k(a)
\end{equation}
which is universal w.r.t. the above property. Phrased differently, if $\cc$ is another
$C^*$-algebra and $\{i_k': k\in\bn\}$ a collection
of $^*$-homorphisms $i_k': \ga\rightarrow\cc$ satisfying \eqref{commrules}, then there must exist
a unique surjective $^*$-morphism $\Psi:\cb\rightarrow\cc$ such that
$i_k'=\Psi\circ i_k$ for all $k$ in $\bn$. The universal property ensures that such a $\cb$ is unique up to
$^*$-isomorphism, and henceforth it will be denoted by $\bigotimes_v^\bn\ga$.
We start by noting that $\bigotimes_v^\bn\ga$ is a quotient of the infinite free product
$\ast_\bn\ga$.
However, in order to show that the embeddings $i_k$ are injective, and thus $\bigotimes_v^\bn\ga$ is non-trivial,
it is convenient to consider an explicit construction of $\bigotimes_v^\bn\ga$, as is done in \cite{FidVic}.
For the reader's convenience, we briefly outline the construction given there.\\
We start by considering a twisted product of  two $G$-graded $C^*$-algebras, $(\ga, G, \gamma)$ and $(\gb, G, \delta)$, $\ga\otimes_v\gb$.
This can be obtained as follows. On the algebraic tensor product $\ga_0\otimes\gb_0$ (here $\ga_0, \gb_0$ are thought of as  vector spaces) of the algebraic layers of $\ga$ and $\gb$ respectively, one may define a $*$-algebra structure by 
\begin{equation}
(a\otimes b)^* :=\overline{v(\partial a, \partial b)}a^*\otimes b^*
\end{equation}
\begin{equation}
(a\otimes b)(c\otimes d) := \overline{v(\partial c, \partial b)}ac\otimes bd
\end{equation}
for homogeneous $a, c$ in $\ga_0$ and $b, d$ in $\gb_0$.\\ 
These operations turn $\ga_0\otimes\gb_0$ into an associative $*$-algebra, \cite[Proposition 5.1]{FidVic}, which we denote by $\ga_0\otimes_v\gb_0$.\\
One can then consider the maximal $C^*$-seminorm  of $\ga_0\otimes_v\gb_0$, which turns out to be
a norm due to the existence of so-called product states, which we recall. Given $\om$ in $\cs(\ga)$ and $\varphi$ in $\cs(\cb)$, with at least one of them being $G$-invariant, then
$$\om\times\varphi\, \left(\sum_{i=1}^n a_i\otimes b_i\right) := \sum_{i=1}^n \om(a_i)\varphi(b_i)$$
defines a normalized positive linear functional on  $\ga_0\otimes_v\gb_0$ (at this level positivity is meant in
an algebraic sense, {\it} i.e.  $\om\times\varphi\,(x^*x)\geq 0$ for all $x$ in $\ga_0\otimes_v\gb_0$).
Then one may consider the GNS representation $\pi_{\om\times\varphi}$ of $\om\times\varphi$, which is bounded
because $\|\pi_{\om\times\varphi}(a\otimes b)\|\leq \|a\|\|b\|$, for all homogeneous $a,b$, thanks to the fact that
the maximal seminorms of $\ga_0$ and $\gb_0$ are the given norm of $\ga$ and $\gb$ respectively. Furthermore, $\om\times\varphi$ is a faithful state
if both $\om$ and $\varphi$ are faithful.
In order to show that the maximal seminorm $\|\cdot\|_{\rm max}$ on  $\ga_0\otimes_v\gb_0$ is actually a norm, it is enough to note that
$$\|x\|_{\rm max}\geq \sup_{\om\in \cs^G(\ga), \varphi\in\cs^G(\gb)} \|\pi_{\om\times \varphi}(x)\|, \, x\in\ga_0\otimes_v\gb_0\,,$$
and $\sup_{\om\in\cs^G(\ga),\varphi\in\cs^G(\gb)}\|\pi_{\om \times \varphi}(x)\|=0$ implies $x=0$, \cite[Proposition 6.3]{FidVic}. Finally, the $C^*$-norm is bounded since $\|\sum_{i=1}^n a_i\otimes b_i\|_{\rm max}\leq \sum_{i=1}^n\|a_i\|\|b_i\|$ for all homogeneous $a_i, b_i$, $i=1, \ldots, n$.
We denote by $\ga\otimes_v\gb$ the completion of $\ga_0\otimes_v\gb_0$ w.r.t. $\|\cdot\|_{\rm max}$.\\
Denote by $i_\ga$ and $i_\gb$ the canonical embeddings of $\ga$ and $\gb$ into $\ga\otimes_v\gb$,  respectively, that is
$i_\ga(a)=a\otimes 1$ $i_\gb(b)=1\otimes b$, $a$ in $\ga$ and $b$ in $\gb$.
By construction, $\ga\otimes_v\gb$ enjoys the universal properties we started our discussion with. 
Now observe that, for every $g$ in $G$, $\gamma_g\otimes \delta_g$ defined on $\ga_0\otimes_v\gb_0$ as
$$\gamma_g\otimes \delta_g\,(\sum_{i=1}^n a_i\otimes b_i) :=\sum_{i=1}^n\gamma_h(a_i)\delta_g(b_i)\,,$$
for homogeneous $a_i, b_i$, is a $*$-isomorphism of $\ga_0\otimes_v\gb_0$, which can be extended to the completion
$\ga\otimes_v\gb$ by maximality of $\|\cdot\|_{\rm max}$. This allows us to conceive of
$\ga\otimes_v\gb$ as a $G$-graded $C^*$-algebra as well.\\
By induction, for ever $n\geq 1$, we can obtain as above the twisted product of $\ga$ with itself
$n$ times, $(\bigotimes_{i=0}^n \ga_i)^v$. Note that, for every $n$, $(\bigotimes_{i=0}^n \ga_i)^v$ 
canonically embeds into $(\bigotimes_{i=0}^{n+1} \ga_i)^v$, which enables us to consider the inductive limit
$C^*$-algebra of the inductive system thus obtained. The inductive limit is an explicit
realization of the universal $C^*$-algebra $\bigotimes_v^\bn\ga$, in which the injectivity of all embeddings
$i_k$, $k$ in $\bn$, is now clear.\\
Finally, we need to establish some terminology. In particular, we need to say what we mean by localized elements of the infinite twisted product $\bigotimes_v^\bn\ga$. We say that  $x$ in  $\bigotimes_v^\bn\ga$ is localized and
its support lies in a finite set $I\subset\bn$ if it belongs to $C^*(i_l(\ga): l\in I)$.  
It is clear that the set of localized elements is a dense
$*$-subalgebra of the infinite twisted product.\\

%

The action of $\bj_\bn$ by endomorphisms on $\ast_\bn\ga$ passes to
$\bigotimes_v^\bn\ga$ because the defining commutation rules \eqref{commrules} are clearly
preserved by all monotone maps.
Unlike $\bj_\bn$, the action of $\bp_\bn$ on the free product does not always  pass to its quotient 
$\bigotimes_v^\bn\ga$. However, if the bicharacter $v$ is antisymmetric, then the action does pass to
the twisted product, as can  be easily checked.\\
Note that (maximal) infinite tensor products of a given $C^*$-algebra, the CAR algebra, and the infinite non-commutative torus are all instances of graded tensor products. The CAR algebra is isomorphic with  the infinite $\bz_2$-graded product of $\mathbb{M}_2(\bc)$, where the grading  is induced by the involutive automorphism $\theta$ given by $\theta(A)=UAU^*$, $A$ in $\mathbb{M}_2(\bc)$, with $U := \left( \begin{array}{cc} 1 &0\\ 0 &-1\\ \end{array} \right)$ and the bicharacter $v$ is given by $v(\chi, \eta) := (-1)^{\chi\eta}$, $\chi, \eta\in\bz_2$. The non-commutative infinite torus with deformation parameter $\a\in\br$ can be obtained as the infinite $\bt$-twisted product of $C(\bt)$ with itself, where $\bt$ acts on $C(\bt)$ by rotation and the bicharacter is $v_\a(k, l) := e^{2\pi i\a kl}$, $k, l$ in $\bz$.


%
%
%
%

\section{Non-commutative Ryll-Nardzewski theorem}
\subsection{A general Ryll-Nardzewski theorem}
In many of the known examples of a quantum stochastic process $\big(\ga,\ch,\{\iota_j\}_{j\in \bn},\xi\big)$,
the variables $\iota_j$ may satisfy  commutation rules. To take but one example,  $\iota_j(\ga)$ and $\iota_k(\ga)$
might commute  for all $j\neq k$. Having commutation rules for the $\iota_j$'s entails considering a (closed) two-sided ideal $I$ of $\ast_\bn\ga$ (in the above example, $I$ is the ideal generated by all commutators
$[i_j(a), i_k(b)]$, $j, k$ in $\bn$ and $a, b$ in $\ga$).\\
 In such a situation, the state $\varphi$ on 
$\ast_\bn\ga$ associated with the process by \eqref{corresp}, namely
\begin{equation*}
\varphi(i_{j_1}(a_1)i_{j_2}(a_2)\cdots i_{j_n}(a_n)):=\langle(\iota_{j_1}(a_1)\iota_{j_2}(a_2)\cdots\iota_{j_n}(a_n)\xi, \xi\rangle
\end{equation*}
for all $n\in\bn$, $j_1\neq j_2\neq\cdots\neq j_n\in \bn$ and
$a_1, a_2, \ldots, a_n\in\ga$, can also be seen as a state of the quotient algebra
$\ast_\bn\ga/ I$, that is there exists a state $\widetilde{\varphi}$
on $\ast_\bn\ga/ I$ such that $\varphi=\widetilde{\varphi}\circ p_I$, where
$p_I: \ast_\bn\ga\rightarrow \ast_\bn\ga/ I$ is the canonical
projection. 
\begin{defin}\label{distri}
We say that the distribution $\varphi$ of a quantum stochastic process $\big(\ga,\ch,\{\iota_j\}_{j\in \bn},\xi\big)$ 
 factorizes through the quotient $\ast_\bn\ga/I$, for some two-sided ideal $I$,  if
$\varphi=\widetilde{\varphi}\circ p_I$ for some state $\widetilde{\varphi}$ on $\ast_\bn\ga/I$, with
$p_I: \ast_\bn\ga\rightarrow \ast_\bn\ga/ I$ being the canonical
projection. 
\end{defin}

The following result provides a satisfactory framework where a general Ryll-Nardzewski holds.
Among others, this framework hosts the classical case, processes with distribution factorizing through tensor products, twisted tensor products by antisymmetric bicharacters, and thus in particular the CAR algebra, and so-called $q$-deformed processes.\\
In the sequel, by an ordered monomial in a quotient $\ast_\bn\ga/ I$ we will
mean an element of the form $p_I(i_{j_1}(a_1)\cdots i_{j_n}(a_n))$, for $n\in\bn$ and  $j_1<\ldots< j_n$  or $j_1>\ldots> j_n$ for $a_i$ in $\ga$, $i=1, \ldots, n$.
\begin{thm}\label{mainRyll}
Let $I$ be a closed two-sided ideal of the infinite free product $\ast_\bn\ga$ such that:
\begin{itemize}
\item [(i)] $I$ is invariant under the action of $\bp_\bn$;
\item [(ii)] the closure of the linear span of all ordered monomials is the whole quotient $\ast_\bn\ga/ I$, 
\end{itemize}
then quantum stochastic processes whose distribution factorizes through $\ast_\bn\ga/ I$
 are spreadable if and only if they are exchangeable.
\end{thm}
Among all possible quotients of the free product, the  $C^*$-subalgebra 
$C^*(\iota_j(\ga): j\in \bn)
\subset\cb(\ch)$ generated by the process itself is a rather natural object to consider.
We have the following corollary.
\begin{cor}\label{Ryllprocesses}
Let $\big(\ga,\ch,\{\iota_j\}_{j\in \bn},\xi\big)$ be a quantum stochastic process. Suppose:
\begin{itemize}
\item [(i)] the natural action of $\bp_\bn$ on $\ast_\bn\ga$ passes to $C^*(\iota_j(\ga): j\in \bn)$;
\item [(ii)] the closure of the linear span of all ordered monomials is $C^*(\iota_j(\ga): j\in \bn)$;
\end{itemize}
Then the process is spreadable if and only if it is exchangeable.
\end{cor}
\begin{rem}\label{tretre}
As for Theorem \ref{mainRyll},
note that if there exists a faithful $\bj_\bn$-invariant state on $\ast_\bn\ga/ I$, whose GNS representation is faithful,
then condition (i) is also necessary for the result to hold.
Indeed, if the state is also $\bp_\bn$-invariant, then by Lemma \ref{quotientlemma} the ideal $I$ is invariant under the action of $\bp_\bn$, which thus
passes to  $\ast_\bn\ga/ I$. 
Similarly, condition (i) of Corollary \ref{Ryllprocesses} is also necessary without even assuming the existence of a spreadable state with faithful GNS representation.
\end{rem}
In order to prove Theorem \ref{mainRyll}, we start by establishing a preliminary result.
Before stating it, we need to further explain the adopted notation. For any
spreadable state  $\varphi $  on the twisted tensor product  $\ast_\bn\ga$
indexed by $\bn$, by a minor abuse of notation we continue to denote by
$E_\varphi$ the orthogonal projection from $\ch_\varphi$ onto the closed subspace
$\{v\in\ch_\varphi: T_h^\varphi v=v,\, \textrm{for all}\, h\in \bj_\bn\}$, where, for every
$h$ in $\bj_\bz$, $T_h^\varphi\in \cb(\ch_\varphi)$ is the (proper) isometry uniquely determined by
$$T_h^{\varphi}  \pi_{\varphi}(a)\xi_\varphi=\pi_{\varphi}(\a_h(a))\xi_\varphi\,\, a\in\ast_\bn\ga\,.$$

\begin{lem}\label{lemproj}
If $\varphi $ is a spreadable state on  $\ast_\bn\ga$, then for all $v, w$ in ${\rm Ran}E_\varphi$
$$\ast_\bn\ga\ni
b\mapsto \langle \pi_{\varphi}(b)v, w \rangle$$ is a $\bj_\bn$-invariant linear functional.\\
Moreover, one has 
$$\langle \pi_{\varphi}(x)\pi_{\varphi}(y)v, w\rangle=\langle \pi_{\varphi}(x)E_{\varphi}\pi_{\varphi}(y)\, v, w\rangle\,,\quad \textrm{for all}\,\, v, w\in\,\,{\rm Ran} E_{\varphi}$$
for all $x=i_{j_1}(a_1)\cdots i_{j_n}(a_n)$ and $y=i_{k_1}(b_1)\cdots i_{k_m}(b_m)$ with 
$j_1<\ldots<j_n< k_1<\ldots <k_m$ or $j_1>\ldots>j_n> k_1>\ldots >k_m$, for all
$n, m, j_1, \ldots, j_n, k_1, \ldots, k_m\in\bn$, $a_1, \ldots, a_n, b_1, \ldots, b_m\in\ga$.
\end{lem}
\begin{proof}
We start by proving that the linear functional  $\ast_\bn\ga\ni
b\mapsto \langle \pi_{\varphi}(b)v, w \rangle$ is $\bj_\bn$-invariant.
 For any $v$ in ${\rm Ran} E_{\varphi}$, we 
 have $$T_h^{\varphi}  \pi_{\varphi}(a)v=\pi_{\varphi}(\a_h(a))v\,,$$ 
 for all $a$ in $\ast_\bn\ga$.
 This can be proved as follows. By cyclicity of $\xi_{\varphi}$, there exists a sequence
 $\{b_n: n\in\bn\}\subset \ast_\bn\ga$ such that $v=\lim_n \pi_{\varphi}(b_n)\xi_{\varphi}$. 
 But then
 \begin{align*}
 T_h^{\varphi} \pi_{\varphi}(a)v&=\lim_n T_h^{\varphi}\pi_{\varphi}(ab_n)\xi_{\varphi}=\lim_n \pi_{\varphi}(\a_h(a))\pi_{\varphi}(\a_h(b_n))\xi_{\varphi}\\
 &=\pi_{\varphi}(\a_h(a))v\,,
 \end{align*}
  in that
 $v=T_h^{\varphi} v=\lim_n T_h^{\varphi} \pi_{\varphi}(b_n)\xi_{\varphi}=\lim_n \pi_{\varphi}(\a_h(b_n))\xi_{\varphi}$.\\
 We are now ready to show our functional is $\bj_\bn$-invariant. For every $b$ in $\ast_\bn\ga$, we have
 \begin{align*}
 \langle \pi_{\varphi}(\a_h(b))v, w \rangle&=\langle T_h^{\varphi} \pi_{\varphi}(b)v, w \rangle=
 \langle \pi_{\varphi}(b)v, (T_h^{\varphi})^*w \rangle\\
 &=\langle  \pi_{\varphi}(b)v, w \rangle\,.
 \end{align*}
 Let now $x, y$ in $\ast_\bn\ga$ as in the statement. We start by dealing with the case when
$j_1<\ldots<j_n< k_1<\ldots <k_m$.
 We apply the Alaoglu-Birkhoff theorem,  see \cite[Proposition 4.3.4]{BR1},  to express $E_{\varphi}$ as a strong limit of a net $\{X_\lambda: \lambda\in \Lambda\}$, with each
 $X_\lambda$ being a finite convex combination of the implementing isometries of maps $h$ in $\bj_\bn$, that is $X_\lambda= \sum_{i=1}^{n_\lambda}\mu_i^\lambda T_{h_i^\lambda}^{\varphi}$, where, for each $\lambda$ in $\Lambda$, $\mu_i^{\lambda}\geq 0$ for all $i=1, \ldots, n_\lambda$, $\sum_{i=1}^{n_\lambda}\mu_i^\lambda=1$,
 and $h_i^\lambda$ lies in $\bj_\bn$. Since we have
 $$
\langle \pi_{\varphi}(x)E_{\varphi}\pi_{\varphi}(y)v, w\rangle=\lim_\lambda\sum_{i=1}^{n_\lambda} \mu_i^{\lambda}\langle \pi_{\varphi}(x)T_{h_i^\lambda}^{\varphi}\pi_{\varphi}(y)v, w\rangle\,,
$$
it is enough to make sure each term $\langle \pi_{\varphi}(x)T_{h_i^\lambda}^{\varphi}\pi_{\varphi}(y)v, w\rangle$ separately  equals 
$\langle \pi_{\varphi}(x)\pi_{\varphi}(y)v, w\rangle$. This can be seen as follows

\begin{align*}
\langle \pi_{\varphi}(x)T_{h_i^\lambda}^{\varphi}\pi_{\varphi}(y)v, w\rangle=\langle \pi_{\varphi}(x)\pi_{\varphi}(\a_{h_i^\lambda}(y))v, w\rangle=
\langle \pi_{\varphi}(xy)v, w\rangle\,
\end{align*}
where the last equality is got to by spreadability of the linear functional $\langle\,\cdot\, v, w \rangle$ considering a monotone map
of $\bn$ to itself which is the identity on the subset $\{j_1, \ldots, j_n\}$ and acts as $h_i^\lambda$ on $\{k_1, \ldots, k_m\}$.
Finally, the other case, {\it i.e.} when the support of $x$ lies to the right of that of $y$, can be reconducted to the first case because
$\langle \pi_{\varphi}(x)E_{\varphi}\pi_{\varphi}(y)v, w\rangle=\langle v, 	\pi_\varphi(y^*)E_\varphi \pi_\varphi(x^*)w\rangle$.

\end{proof}

\begin{proof}[Proof of Theorem \ref{mainRyll}]
We only need to prove that a spreadable process with distribution $\varphi$ is exchangeable as well.
Since the permutation group is generated by transpositions, it is enough to verify that $\varphi$ is invariant
under transpositions.
We start by proving this invariance for monomials of the type
$i_k(a)i_{k+1}(b)$ for $k\in\bn$, $a, b\in\ga$  with respect to the transposition that switches
$k$ and $k+1$.
On the one hand, thanks to Lemma \ref{lemproj} we have
\begin{align*}
\varphi(i_k(a)i_{k+1}(b))&=\overline{ \varphi(i_{k+1}(b^*)i_k(a^*))}=\overline{\langle\pi_\varphi(i_{k+1}(b^*))E_\varphi \pi_\varphi(i_k(a^*))\xi_\varphi,\xi_\varphi\rangle}\\
&=\sum_{i\in I}\overline{\langle\pi_\varphi(i_k(a^*)\xi_\varphi, v_i \rangle}\overline{\langle\pi_\varphi(i_{k+1}(b^*)v_i, \xi_\varphi \rangle}\\
&=\sum_{i\in I} \langle \pi_\varphi(i_k(a))v_i,\xi_\varphi \rangle\langle\pi_\varphi(i_{k+1}(b))\xi_\varphi, v_i \rangle\,,
\end{align*}
where $\{v_i:i\in I\}$ is an orthonormal basis of ${\rm Ran} E_\varphi$.\\
On the other hand, we have
\begin{align*}
\varphi(i_{k+1}(a)i_k(b))&=\langle \pi_\varphi(i_{k+1}(a))E_\varphi\pi_\varphi(i_k(b))\xi_\varphi, \xi_\varphi \rangle\\
&=\sum_{i\in I}\langle \pi_\varphi(i_k(b))\xi_\varphi, v_i \rangle\langle\pi_\varphi(i_{k+1}(a)v_i, \xi_\varphi \rangle\\
&=\sum_{i\in I}\langle \pi_\varphi(i_{k+1}(b))\xi_\varphi, v_i \rangle\langle\pi_\varphi(i_k(a)v_i, \xi_\varphi \rangle\,
\end{align*}
where the last equality is due to the fact that the linear functionals $\langle \, \cdot \, \xi_\varphi, v_i \rangle$
and $\langle \, \cdot \, v_i, \xi_\varphi \rangle$ are all spreadable (and {\it a fortiori} shift-invariant).
Comparing the two expressions, we see $\varphi(i_k(a)i_{k+1}(b))=\varphi(i_{k+1}(a)i_k(b))$.\\
By the uniqueness of the Hanh-Jordan decomposition of a bounded linear functional, it follows that every
$\bj_\bn$-invariant linear functional is a linear combination of four $\bj_\bn$-invariant states. Therefore,
any such linear functional $f$ still satisfies $f(i_k(a)i_{k+1}(b))= f(i_{k+1}(a)i_k(b))$, for all $k$ in $\bn$ and
$a, b$ in $\ga$.\\
We are now ready to prove the invariance of $\varphi$ under the transposition that switches $k$ and $k+1$ on general monomials of the type $xi_k(a)i_{k+1}(b)y$, where $x=i_{j_1}(x_1)\cdots i_{j_l}(x_l)$, $j_1<j_2<\ldots <j_l<k$, $x_1, \ldots, x_l$ in $\ga$,
 and $y=i_{s_1} (y_1)\cdots i_{s_m} (y_m)$, $k+1<s_1<s_2<\ldots < s_m$, $y_1, \ldots, y_m$ in $\ga$ (terms
ordered in the reverse order can be dealt with analogously). We have
\begin{align*}
\varphi(x i_k(a)i_{k+1}(b)y) 
&= \langle \pi_\varphi(xi_k(a)i_{k+1}(b)y)\xi_\varphi, \xi_\varphi \rangle\\
&= \langle \pi_\varphi(x)\pi_\varphi(i_k(a)i_{k+1}(b))E_\varphi \pi_\varphi(y)\xi_\varphi, \xi_\varphi \rangle\\
&= \sum_{i\in I} \langle \pi_\varphi(x)\pi_\varphi(i_k(a)i_{k+1}(b))v_i, \xi_\varphi \rangle \langle \pi_\varphi(y)\xi_\varphi, v_i \rangle\\
&= \sum_{i\in I} \langle \pi_\varphi(x)E_\varphi\pi_\varphi(i_k(a)i_{k+1}(b))v_i, \xi_\varphi \rangle \langle \pi_\varphi(y)\xi_\varphi, v_i \rangle\\
&= \sum_{i\in I} \sum_{j\in I} \langle \pi_\varphi(i_k(a)i_{k+1}(b))v_i, v_j \rangle \langle \pi_\varphi(x)v_j, \xi_\varphi \rangle \langle \pi_\varphi(y)\xi_\varphi, v_i \rangle\\
&= \sum_{i\in I} \sum_{j\in I} \langle \pi_\varphi(i_{k+1}(a)i_{k}(b))v_i, v_j \rangle \langle \pi_\varphi(x)v_j, \xi_\varphi \rangle \langle \pi_\varphi(y)\xi_\varphi, v_i \rangle\\
\end{align*}
 where the last equality holds because the linear functionals $\langle \, \cdot \, v_i, v_j \rangle$ are  $\bj_\bn$-invariant, with
 $\{v_i:i\in I\}$ being an orthonormal basis of ${\rm Ran} E_\varphi$.
  Now, by making the same computation in the reverse order we find
 \begin{align*}
 \sum_{i\in I} \sum_{j\in I} \langle \pi_\varphi(i_{k+1}(a)i_{k}(b))v_i, v_j \rangle \langle \pi_\varphi(x)v_j, \xi_\varphi \rangle \langle \pi_\varphi(y)\xi_\varphi, v_i \rangle=
 \varphi(x i_{k+1}(a)i_k(b)y)\,.
\end{align*}
Since the quotient $\ast_\bn\ga/ I$ is  the closure of the linear span of ordered monomials and $I$ is invariant under $\bp_\bn$, we finally see that $\varphi$ is exchangeable.
\end{proof}

An immediate consequence of Theorem \ref{mainRyll}
is the following.
\begin{cor}\label{Rylltwisted}
Let $\big(\ga,\ch,\{\iota_j\}_{j\in \bn},\xi\big)$ be a quantum stochastic process, whose distribution $\varphi$ factorizes through 
 $\bigotimes_v^\bn\ga$, with $v$ being antisymmetric.
 Then, the process is spreadable if and only if it is exchangeable.
\end{cor}

\subsection{Counterexamples: spreadable processes which are not exchangeable} \label{SecBraid}
The present section is devoted to showcasing a wide class of quantum stochastic processes that are spreadable yet not exchangeable.   Product states on twisted tensor products turn out to be
a key tool to provide such examples.\\
To begin with, we characterize product states on twisted tensor products. To this end,
we start by defining what we call the spectral support of a state on the sample $\ga$.
Given a state $\om$ on a $G$-graded $C^*$-algebra $(\ga, G, \gamma)$, its spectral support is the set 
$${\rm supp}_G\,\om :=\{\eta\in\widehat{G}: \exists\, a\in\ga_\eta\,\,	{\rm with}\,\, \om(a)\neq 0\}\, .$$
Note that $\om$ is $G$-invariant if and only if ${\rm supp}_G\,\om$ is trivial.

\begin{thm}\label {existenceprod}
Let $\om$ be a state on $\ga$. There exists a (unique) state on $\ga\otimes_v\ga$, $\om\times\om$, such that
$\om\times\om\,(a\otimes b)=\om(a)\om(b)$ for all $a, b\in\ga$ if and only if for any
$\eta_1, \eta_2$ in ${\rm supp}_G\,\om $ one has $v(\eta_1, \eta_2)=1$.
\end{thm}

\begin{proof}
We start by proving that the equality $v(\eta_1, \eta_2)=1$, for $\eta_1, \eta_2$ in ${\rm supp}_G\,\om $ is a necessary
condition for the state $\om\times\om$ to exist.
Let $\eta_1, \eta_2$ in ${\rm supp}_G\,\om $ and $a, b$ in $\ga$ such that $\om(a)\neq 0$, $\om(b)\neq 0$ and
$\partial a=\eta_1$, $\partial b=\eta_2$. By hermitianity of $\om\times\om$, we have
\begin{align*}
\om(a)\om(b)&=\om\times\om\,(a\otimes b)=\overline{\om\times\om((a\otimes b)^*)}=
\overline{\om\times\om   (\overline{v(\partial a, \partial b)} a^*\otimes b^*)}\\
&=v(\partial a, \partial b ) \overline{\om(a^*)\om(b^*)}=v(\partial a, \partial b )\om(a)\om(b)\,
\end{align*}
hence $v(\eta_2, \eta_1)=1$, as $\om(a)\om(b)$ is different from $0$. \\
For the sufficiency, we start by observing that the condition implies
${\rm supp}_G\,\om  \subset \Delta_v$. Second, the subgroup $H_\om$ of $\widehat{G}$ generated by
${\rm supp}_G\,\om $ is a subset of $\Delta_v$. Let us consider its annihilator
$G_\om := H_\om^\perp:=\{g\in G: \chi(g)=1\,\,\textrm{for all}\,\, \chi\in H_\om\}\subset G$.  \\
The conditional expectation
$$F(x) := \int_{G_\om\times G_\om} \a_g\otimes\a_h (x)\, {\rm d}\mu_{G_\om\times G_\om}\, ,x\in\ga\otimes_v\ga\,,$$ projects onto the subalgebra of those elements fixed by all automorphisms
$\a_g\otimes\a_h $, $(g, h)$ in $G_\om\times G_\om$. We claim that these are exactly the closed linear span
of elements of the type $a\otimes b$, with homogeneous $a, b$ in $\ga$ and $\partial a, \partial b$ in $H_\om$.
 Indeed, it is clear that $F(x)=x$ is $x=a\otimes b$ with
$a, b\in\ga$  homogeneous and $\partial a, \partial b$ are in $H_\om$. 
In addition, $F(x)=0$ for all $x=a\otimes b$, when $a, b\in\ga$  are homogeneous
and one of $\partial a, \partial b$ does not sit in $G_\om^\perp=H_\om^{\perp\perp}=H_\om$.
This can be seen as follows. For such an $x$, $F(x)=\int_{G_\om} \partial a(g){\rm d}\mu_{G_\om}(g)\int_{G_\om} \partial b(g){\rm d}\mu_{G_\om}(g)\, a\otimes b$. If for example $\partial a$ does not fit in $G_\om^\perp$, then
the restriction of $\partial a$ to $G_\om$ is a non-trivial character of $G_\om$, and therefore  
$\int_{G_\om} \partial a(g){\rm d}\mu_{G_\om}(g)=0$, which implies $F(x)=0$.\\
Since the restriction of $v$  to $H_\om\times H_\om$ is $1$, the latter subalgebra identifies to some completion $\ga^{G_\om}\overline{\otimes} \ga^{G_\om}$ of the algebraic tensor product
tensor product of $\ga^{G_\om}$ with itself, which projects onto the minimal tensor
product $\ga^{G_\om}\otimes_{\rm min}\ga^{G_\om}$. In particular, for any given state
$\om$ on $\ga^{G_\om}$, there exists a unique state $\om\otimes\om$ on $\ga^{G_\om}\overline{\otimes} \ga^{G_\om}$ such that $\om\otimes\om(a\otimes b)=\om(a)\om(b)$ for all $a, b$ in $\ga^{G_\om}$.
 It is now clear how to define the product state, for  the state
$\om\times \om := (\om\otimes\om)\circ F$ has the required properties.
\end{proof}
The characterization given in Theorem \ref{existenceprod} also covers the case of an infinite twisted product. Indeed, with the same argument we have the following.
\begin{cor}
Let $\om$ be a state on $\ga$. There exists a (unique) state $\times\om $ on $\bigotimes_v^\bn\ga$,
such that
$$\times\om(i_{j_1}(a_1)\cdots i_{j_n}(a_n))= \om(a_1)\cdots\om(a_n)\,,$$
for all $n$ in $\bn$, $j_1<\ldots<j_n$ in $\bn$, $a_1, \cdots, a_n$ in $\ga$, if and only if for any
$\eta_1, \eta_2$ in ${\rm supp}_G\,\om $ one has $v(\eta_1, \eta_2)=1$.
\end{cor}
We denote by $\cs_v(\ga)$ the set all states $\om$ such that the infinite product $\times\om$ exists, namely
$$\cs_v(\ga):=\{\om\in\cs(\ga): v\upharpoonright_{{\rm supp}_G\,\om \times{\rm supp}_G\,\om }=1\}\,.$$
Note that $\cs_v(\ga)$ is a closed subset of $\cs(\ga)$ in the $*$-weak topology.\\

The next result shows that product states are a source of spreadable quantum processes that fail to be 
exchangeable as soon as the bicharacter is not antisymmetric. If the assigned grading on the sample algebra  is ergodic, the resulting situation is quite surprising, in that no exchangeable processes whose distribution
factorizes through the corresponding twisted tensor product exist.

\begin{prop} \label{counterexamples}
Let $p:\ast_\bn\ga\rightarrow \otimes_{v}^\bn\ga$ the quotient projection.
\begin{enumerate}
\item[(i)] Suppose $\exists \chi,\eta\in\hat{G}$ such that $v(\chi,\eta)\neq \overline{v(\eta,\chi)}$ and that $\ga_\chi\neq\{0\}, \ga_\eta\neq\{0\}$. Then $\exists\, \omega\in \cs_v(\ga)$ such that $\times \omega \circ p$ is the distribution of a spreadable but non-exchangeable quantum stochastic process.
\item[(ii)] Suppose the action of $G$ is faithful and topologically ergodic and that the bicharacter $v$ is not antisymmetric. Then, for any $\omega\in\cs_v(\ga)$, $\times \omega \circ p$ is the distribution of a spreadable but non-exchangeable quantum stochastic process. Moreover, the set $\{\varphi\in\cs(\otimes_{v}^\bn\ga): \varphi\circ p \textrm{ is exchangeable}\}$ is empty.
\end{enumerate}
\end{prop}
\begin{proof}
To prove (i), let $a_\chi,a_\eta$ be nonzero elements of $\ga_\chi,\ga_\eta$ respectively and let $C^*(a_\chi,a_\eta)$ the unital, separable $C^*$-algebra generated by $a_\chi,a_\eta$. Let $\widetilde{\omega}$ be a faithful state on  $C^*(a_\chi,a_\eta)$ and $\check{\omega}\in\cs(\ga)$ be an extension of $\widetilde{\omega}$ which exists by Segal's extension theorem. Finally, let $\omega:=\check{\omega}\circ E_G$ where $E_G$ is the canonical conditional expectation given by the action of $G$. Note that $\omega\in\cs_v(\ga)$ and that $\omega(a_\chi a^*_\chi)=\widetilde{\omega}(a_\chi a^*_\chi)\neq 0$ and $\omega(a_\eta a^*_\eta)=\widetilde{\omega}(a_\eta a^*_\eta)\neq 0$ since $a_\chi a^*_\chi,a_\eta a^*_\eta\in\ga^G$.  We have
$$\times\omega(i_1(a_\chi)i_2(a_\eta)i_1(a^*_\chi)i_2(a^*_\eta))=v(\chi,\eta)\omega(a_\chi a^*_\chi)\omega(a_\eta a^*_\eta),$$
while
$$\times\omega(i_2(a_\chi)i_1(a_\eta)i_2(a^*_\chi)i_1(a^*_\eta))=\overline{v(\eta,\chi)}\omega(a_\chi a^*_\chi)\omega(a_\eta a^*_\eta).$$
and thus $\times\omega$ is not exchangeable.
\\
To prove (ii), since $v$ is not antisymmetric, we can pick $\chi,\eta\in \hat{G}$ such that $v(\chi,\eta)\neq \overline{v(\eta,\chi)}$ and by the proof of Proposition \ref{fullspectrum} we have that there are unitaries $U_\chi\in\ga_\chi, U_\eta\in\ga_\eta$. Thus, for any $\omega\in\cs_v(\ga)$, we can compute
$$\times \omega(i_1(U_\chi)i_2(U_\eta)i_1(U^*_\chi)i_2(U^*_\eta))=v(\chi,\eta),$$
while
$$\times \omega(i_2(U_\chi)i_1(U_\eta)i_2(U^*_\chi)i_1(U^*_\eta))=\overline{v(\eta,\chi)}.$$
Finally, the last part follows by completely analogous computations.
\end{proof}

\begin{rem}
A class of examples satisfying (ii) of Proposition \ref{counterexamples} is obtained by taking $\ga:=C(G)$ with natural action of $G$ and $v:\hat{G}\times\hat{G}\rightarrow\bt$ any bicharacter which is not antisymmetric.
\end{rem}

Working in the formalism provided by von Neumann algebras,  Gohm and  K\"{o}stler developed in \cite{GK08}
an approach to a new  symmetry which they call braidability, see \cite[Definition 0.1]{GK08}, and  which they show to be 
intermediate between exchangeability and spreadability, see \cite[Theorem 0.2]{GK08}.
The intuition behind their approach is that the role played by the symmetric group and its natural actions should in the non-commutative realm be taken by the braid group. Bearing this in mind, it is then rather natural to ask
whether a theorem \`a la Ryll-Nardzewski may hold in which  braidability is substituted for exchangeability. This is exactly the question that  Evans, Gohm and  K\"{o}stler raise in \cite{EGK17}, namely  if spreadable processes are the same as braidable processes.\\
Strictly speaking, 
 the notion of braidability as given in \cite{GK08} is not a distributional symmetry as this is usually understood, in that
 the braid group does not act on distributions but rather on the non-commutative probability space where the process is realized.\\
In the $C^*$-algebraic formalism the present paper is set in, the distribution of a process is 
essentially a state on a suitable quotient of the infinite free product of the sample algebra, {\it cf.} Definition \ref{distri}.
Importantly, the symmetry group/semigroup may act on the quotient even if it does not act naturally on the free product, thus acting by duality on the states of
the quotient algebra as well (i.e. acting on distributions). A case in point illustrating this scenario even in the framework of Classical Probability is provided by 
rotatability, with the orthogonal group naturally  acting on  the tensor products of the unitalization of $C_0(\br)$,
while  failing to do so on the associated free products. Another, in the non-commutative setting, is again provided by
rotatability, which makes perfect sense for states on the CAR algebra by second quantization, although no natural actions of the
orthogonal/unitary group can be considered on the infinite free product of $\mathbb{M}_2(\bc)$ with itself.\\
That said, we believe the following is a natural formulation of what ought to be meant by braidability, as a distributional symmetry,  in the $C^*$-algebra formalism we work in.
Before we give our definition, we first recall what braid group we are dealing with.\\
Let  
$$
\bb_\bn :=
\left\langle \sigma_1, \sigma_2, \dots  \, | \, \begin{array}{l}
\sigma_i \sigma_j = \sigma_j \sigma_i \text{ for } |i-j|>2\\
\sigma_i \sigma_{i+1}\sigma_i = \sigma_{i+1} \sigma_i \sigma_{i+1} \text{ for all } i\geq 1
\end{array}
\right\rangle
$$
be  Artin's braid group on $\bn$. Note that $\bb_\bn$ projects onto $\bp_\bn$.
The second  group of defining relations of Artin's group is sometimes referred to as the
Yang-Baxter relation.\\

\begin{defin}\label{Cbraid}
Let $I$ be a  closed two-sided ideal of $\ast_\bn\ga$.
A quantum stochastic process $\big(\ga,\ch,\{\iota_j\}_{j\in \bn},\xi\big)$ whose distribution factorizes through the quotient
$\ast_\bn\ga/I$ (see Definition \ref{distri}) is said to be $I$-\emph{braidable} if there exists an action of the braid group, $\rho: \mathbb{B}_\bn\to \text{Aut}(\ast_\bn\ga/I)$, which leaves the state on $\ast_\bn\ga/I$ induced by $\xi$ invariant and such that the following properties hold
\begin{align}
i_n' &= \rho(\sigma_n \sigma_{n-1} \cdots \sigma_1)\circ i_0' \qquad \text{for all } n \geq 1\label{braid1} \\
i_0' &= \rho(\sigma_n)\circ i_0' \qquad \text{if } n \geq 2\,, \label{braid2}
\end{align}
where $i_n':= p_I\circ i_n$.
\end{defin}

A key property of the above definition is that the braid group is supposed to act directly on the quotient
$\ast_\bn\ga/I$, without assuming the existence of an action on the free product compatible with the ideal $I$. Interesting examples of braidable states in the above sense are studied in \cite{BJLW}, where the quotient
is the so-called parafermion algebra $PF_\infty$, see also Section \ref{parafermion} of the present paper.\\
As one can imagine, $I$-braidability implies spreadability, as can be proved by following the lines of \cite[Theorem 2.2]{GK08}.
Furthermore, when $I$ is compatible with the natural action of $\bp_\bn$,  exchangeability clearly implies $I$-braidability.
In particular, under the hypotheses of Theorem \ref{mainRyll}, spreadability, $I$-braidability and exchangeability are the same.\\
 One then might wonder if under suitable working hypotheses spreadability and braidability are still the same even when $I$ fails to be compatible with the permutation action.\\
The next result shows that this may not be the case. Indeed, 
what we do is  we provide a counterexample where the quotient is a twisted product, which thus fulfills the hypothesis $(ii)$ of Theorem \ref{mainRyll}
 and  is naturally acted upon by $\bj_\bn$.\\
To discuss our counterexample, we first establish some notation.
 Denote by $\mathbb{A}_\alpha$ the non-commutative $2$-torus with deformation parameter $\alpha\in \br$, namely the universal $C^*$-algebra generated by two unitaries $u,w$ such that  $u w=e^{2\pi i\alpha}wu$. Denote by ${\rm tr}$ the canonical trace on $\mathbb{A}_\alpha$, namely the state determined by ${\rm tr}(u^iw^j)=\delta_{i,0}\delta_{j,0}$.
Let $\bigotimes_v^{\bn }\mathbb{A}_\alpha$ be the twisted tensor product, where $G:=\bt^2$ acts by rotations on $\mathbb{A}_\alpha$, with bicharacter $v:\bz^2\times\bz^2\rightarrow \bt$, and let $\times {\rm tr}$ be its state obtained as infinite product of ${\rm tr}$ with itself ($\times {\rm tr}$ is still a trace).
Clearly, $\bigotimes_v^{\bn }\mathbb{A}_\alpha$ is a quotient of $\ast_\bn\mathbb{A}_\alpha$ by the ideal $I_v$ generated by  elements
of the form $i_l(a)i_k(b)- v( \deg(a),\deg(b))i_k(b)i_l(a)$ with $a, b$  being homogeneous and $l, k$ varying in $\bn$ with $l<k$.
 We simply denote by $p$ the corresponding quotient map.

\begin{thm}\label{counterexamplebraid}
Let $v$ be the bicharacter given by  $v(k, l) := e^{2\pi i(\theta k_1l_1+\theta  k_2l_2)}$, $k=(k_1, k_2), l=(l_1, l_2)$ in $\bz^2$, 
and let $\theta $, 
$1$, $\alpha$ be rationally independent. Then $\times {\rm tr}\circ p$ is the distribution of a quantum stochastic process factorizing
through the quotient $\ast_\bn \mathbb{A}_\alpha/I_v$  which is spreadable but not $I_v$-braidable.
 \end{thm}
\begin{proof}
We will argue by contradiction  by supposing that
there does exist an automorphic action $\rho$ of the braid group $\mathbb{B}_\bn$ on
$\pi_{\times{\rm tr}}(\bigotimes_v^{\bn }\mathbb{A}_\alpha)$ as in Definition \ref{Cbraid} (note that
$\pi_{\times{\rm tr}}(\bigotimes_v^{\bn }\mathbb{A}_\alpha)$ is isomorphic with 
$\bigotimes_v^{\bn }\mathbb{A}_\alpha$ as $\times{\rm tr}$ is faithful).\\
To ease the notation, the automorphims $\rho(\sigma_j)$ will be simply
denoted by $\sigma_j$ in what follows.\\
We first show that Definition \ref{Cbraid} completely determines $\sigma_k \circ \iota_j$ when $k\neq j$. Indeed, from the relations \eqref{braid1}, \eqref{braid2} it follows easily that 
$\sigma_k \circ \iota_j=\iota_j$ when $k\geq j+2$. From relation \eqref{braid1}, $\sigma_k \circ \iota_{k-1}=\iota_{k}$. Finally, from the Yang-Baxter relation of the braid group and again from relations  \eqref{braid1}, \eqref{braid2} , for $k\leq j-1$ we get
\begin{align*}
\sigma_k \circ \iota_j &=\sigma_k \circ \sigma_j\circ \sigma_{j-1}\circ\cdots \circ \sigma_1\circ\iota_0\\
&=\sigma_j\circ \sigma_{j-1}\circ\cdots\circ\sigma_{k+2}\circ\sigma_{k}\circ \sigma_{k+1}\circ \sigma_{k} \circ \cdots\circ\sigma_1\circ \iota_0\\
&=\sigma_j\circ \sigma_{j-1}\circ\cdots\circ\sigma_{k+2}\circ\sigma_{k+1}\circ \sigma_{k}\circ \sigma_{k+1} \circ \cdots\circ\sigma_1\circ \iota_0=\iota_j.
\end{align*}
Thus it only remains to find the possible values of $\sigma_{j} \circ \iota_{j}$.\\
We denote by $u_l=\iota_l(u), v_l=\iota_l(w)$ the generators of $\iota_l(\mathbb{A}_\alpha)$ with $u_l v_l=e^{2\pi i\alpha}v_lu_l$.
Note that $u_l v_j=v_j u_l$ for $l\neq j$.\\
Denote by $\xi$ the GNS vector of the GNS representation of $\bigotimes_v^{\bn}\mathbb{A}_\alpha$ w.r.t. $\times {\rm tr}$ and by
 $\ch$ 
 the corresponding Hilbert space.\\
Note that the countable subset of $\ch$   
$$\{\xi\}\cup\{u_{i_1}^{j_{i_1}}v_{i_1}^{q_{i_1}}\cdots u_{i_k}^{j_{i_k}}v_{i_k}^{q_{i_k}}\xi: k\in\bn,  i_1<\ldots<i_k, j_{i_1}, \ldots, j_{i_k}, q_{i_1},\ldots, q_{i_{k}}\in\bz\}$$
provides an orthonormal basis of $\ch$. For convenience, we adopt the convention that, for every $l=1, 2, \ldots, k$,
at least one between the exponents $j_{i_l}$ and $q_{i_l}$ is different from zero.\\
Consider the vector $\sigma_1(u_1)\xi$:  since it is different from zero (and $\s_1(u_1)\xi\neq \xi$ as well since $\xi$ is a separating vector), we can safely suppose
that 
$$\langle \sigma_1(u_1)\xi, u_{i_1}^{j_{i_1}}v_{i_1}^{q_{i_1}}\cdots u_{i_k}^{j_{i_k}}v_{i_k}^{q_{i_k}}\xi \rangle \neq 0\,, $$
for a fixed basis vector $u_{i_1}^{j_{i_1}}v_{i_1}^{q_{i_1}}\cdots u_{i_k}^{j_{i_k}}v_{i_k}^{q_{i_k}}\xi$.\\
More precisely, we claim that the only non-zero Fourier coefficients  of $\sigma_1(u_1)\xi$ are those corresponding $\xi$ or to basis vectors 
with $i_k=0$ or $i_k=1$.\\
We can show this by contradiction. Suppose $\langle \sigma_1(u_1)\xi, u_{i_1}^{j_{i_1}}v_{i_1}^{q_{i_1}}\cdots u_{i_k}^{j_{i_k}}v_{i_k}^{q_{i_k}}\xi \rangle \neq 0$ for some basis vector with $i_k\geq 2$ (we recall that
$i_0<i_1<\ldots<i_k$ and
at least one between $j_{i_k}$ and $q_{i_k}$ is  different from zero).
Note that for all $r\geq 2$
the following relations hold:
$$\sigma_1(u_ru_1)= u_r\sigma_1(u_1)=e^{-2 \pi i\theta } \sigma_1(u_1u_r)= e^{-2\pi i\theta }\sigma_1(u_1)u_r,$$
 that is 
\begin{equation}\label{rel3}
\sigma_1(u_1)=e^{2\pi i\theta } u_r\sigma_1(u_1)u_r^*.
\end{equation}
In particular, we must have
\begin{align}\label{equalcoeff}
&\langle \sigma_1(u_1)\xi, u_{i_1}^{j_{i_1}}v_{i_1}^{q_{i_1}}\cdots u_{i_k}^{j_{i_k}}v_{i_k}^{q_{i_k}}\xi\rangle=\nonumber \\ 
&\langle e^{2\pi i\theta } u_r\sigma_1(u_1)u_r^*\xi, u_{i_1}^{j_{i_1}}v_{i_1}^{q_{i_1}}\cdots u_{i_k}^{j_{i_k}}v_{i_k}^{q_{i_k}}\xi\rangle\, .
\end{align}
On the other hand, we also have
$$u_r^*u_{i_1}^{j_{i_1}}v_{i_1}^{q_{i_1}}\cdots u_{i_k}^{j_{i_k}}v_{i_k}^{q_{i_k}}u_r=e^{-i 2\pi\theta  (\sum_{i_p<r} (-j_{i_p}) + \sum_{i_p>r}j_{i_p})- 2 \pi i \alpha q_r} u_{i_1}^{j_{i_1}}v_{i_1}^{q_{i_1}}\cdots u_{i_k}^{j_{i_k}}v_{i_k}^{q_{i_k}}\,.$$ 
The expression in \eqref{equalcoeff} can be calculated explicitly by means of the above equality and exploting the trace property of the vector state associated with $\xi$, yielding:
\begin{align*}
&\langle \sigma_1(u_1)\xi, u_{i_1}^{j_{i_1}}v_{i_1}^{q_{i_1}}\cdots u_{i_k}^{j_{i_k}}v_{i_k}^{q_{i_k}}\xi \rangle=\\
&\langle \sigma_1(u_1)\xi, u_{i_1}^{j_{i_1}}v_{i_1}^{q_{i_1}}\cdots u_{i_k}^{j_{i_k}}v_{i_k}^{q_{i_k}}\xi \rangle e^{ 2\pi i\theta  (\sum_{i_p<r} (-j_{i_p}) + \sum_{i_p>r}j_{i_p})+2 \pi  i\alpha q_r+2\pi i \theta}\,,
\end{align*}
hence
\begin{align}\label{i_r}
&\theta (\sum_{i_p<r}(-j_{i_p}) + \sum_{i_p>r}j_{i_p}) + \alpha q_r=-\theta + \ell,
\end{align}
for some $\ell\in\bz$, which, by rational independence, implies $q_r=0$ and
\begin{align}\label{i_42}
&\sum_{i_p<r}(-j_{i_p}) + \sum_{i_p>r}j_{i_p}=-1.
\end{align}
 Subtracting relation \eqref{i_42} with $r=i_{k}+1$ and $r={i_k}$, we get
\begin{align*}
&0=\sum_{i_p<{i_{k}+1}}(-j_{i_p}) + \sum_{i_p>i_{k}+1}j_{i_p}- \left( \sum_{i_p<i_k}(-j_{i_p}) + \sum_{i_p>i_k}j_{i_p} \right)\\
&=\sum_{i_p<{i_{k}+1}}(-j_{i_p})  -  \sum_{i_p<i_k}(-j_{i_p})=- j_{i_k}
\end{align*}
hence $j_{i_k}=0$, which is  absurd as $q_{i_k}=0$ as well, and at the beginning we assumed that not both exponents vanished. This ends the proof of the claim.\\
Before going on, note that $\xi$ and vectors of the basis with $i_k=0, 1$ can also be collectively obtained as $u_0^{j_0}v_0^{q_0}u_1^{j_1}v_1^{q_1}\xi$, if one allows the exponents $j_0, q_0, j_1, q_1$ to vanish, which is what we do to unify the following calculations.\\
Similarly to what we saw above, with   
$u_1\sigma_1(u_1)=\sigma_1(u_0u_1)=e^{i2\pi\theta } \sigma_1(u_1u_0)= e^{i2\pi\theta }\sigma_1(u_1)u_1$ we get 
$\sigma_1(u_1)=e^{-2\pi i\theta } u_1\sigma_1(u_1)u_1^*$, which, as before, implies the following equality involving the exponents $j_l$ and $q_{l}$, $l=0, 1$, of the vectors $u_0^{j_0}v_0^{q_0}u_1^{j_1}v_1^{q_1}\xi$ with $\langle \sigma_1(u_1)\xi, u_0^{j_0}v_0^{q_0}u_1^{j_1}v_1^{q_1}\xi \rangle\neq 0$:
\begin{align}\label{i_2bis}
&-\theta j_{0} + q_1\alpha=\theta +\ell
\end{align}
for $\ell\in\bz$.  In order for Equation \eqref{i_2bis} to hold, we must have   $q_1=0$, 
and $j_0=-1$.  Similarly, by \eqref{rel3}  for any $r\geq 2$ we have 
 $$\langle \sigma_1(u_1)\xi, u_0^{-1}v_0^{q_0}u_1^{j_1}\xi \rangle=e^{2\pi i \theta}\langle u_r \sigma_1(u_1)u_r^*\xi, u_0^{-1}v_0^{q_0}u_1^{j_1}\xi \rangle\,,$$
which, proceeding analogously as before,  gives  $j_1=2$.\\
Finally, consider 
$ v_1\sigma_1(u_1)=\sigma_1(v_0u_1)= \sigma_1(u_1v_0)=\sigma_1(u_1)v_1$, that is 
$\sigma_1(u_1)=v_1\sigma_1(u_1)v_1^*$. As before, equating $\langle \sigma_1(u_1)\xi, u_0^{-1}v_0^{q_0}u_1^{j_1}\xi  \rangle$ and
 $\langle v_1\sigma_1(u_1)v_1^*\xi, u_0^{-1}v_0^{q_0}u_1^{j_1}\xi \rangle$,
we find
\begin{align}\label{i_3bis}
&-\theta q_0 -j_1\alpha=\ell\,,
\end{align}
for $\ell\in\bz$. Rational independence   implies $j_1=0$ and this is a contradiction, which shows that it is not possible to define $\sigma_1$ in a way that  makes it compatible with Relations \eqref{braid1}-\eqref{braid2} and making $\times \rm{tr}$ invariant.
\end{proof}
\begin{rem}\label{remark_GK}
The above result   shows the existence of a spreadable state (process) which is not braidable  in the sense of 
 Gohm and  K\"{o}stler, {\it cf.} Definition \cite[0.1]{GK08}, at least if their definition is to be taken strictly. Indeed, it is enough to consider $\pi_{\times{\rm tr}}(\bigotimes_v^{\bn }\mathbb{A}_\alpha)''$,  the 
von Neumann algebra generated by the GNS representation of $\times {\rm tr}$. Since
$\xi_{\times {\rm tr}}$ is a separating vector for that von Neumann algebra (because $\times {\rm tr}$ is a trace, and traces have central support),
the proof above carries over showing that no action of the braid group can be given on $\pi_{\times{\rm tr}}(\bigotimes_v^{\bn }\mathbb{A}_\alpha)''$  
making the faithful vector state associated with $\xi_{\times{\rm tr}}$ braidable. 
Nevertheless, our result does not rule out the 
existence of a braidable process (in  Gohm and  K\"{o}stler's sense)
which has the same distribution as the process dealt with in Theorem \ref{counterexamplebraid}. \\ 
For these reasons Theorem \ref{counterexamplebraid} may be seen as a first insight into the full solution
to the problem raised in \cite{EGK17} whether spreadability is equivalent to the braidability in the sense of
 Gohm and  K\"{o}stler.

\end{rem}


%

\begin{rem}
On the infinite non-commutative torus $\mathbb{A}_\theta^\bn$, that is the universal $C^*$-algebra generated by a countable set of unitaries $\{u_j: j\in\bn\}$ such that $u_ju_k=e^{2\pi i\theta}u_ku_j$ for all $j<k$, with $\theta$ irrational (note that $\mathbb{A}_\theta^\bn\cong \ast_\bn C(\bt)/ I_\theta$, with $I_\theta$ being the two-sided ideal corresponding to the commutation rules of the infinite-torus) there is only one   automorphic action
of $\mathbb{B}_\bn$ making the quadruple $(C(\bt),  \{\iota_j: j\in\bn\}, \ch, \xi)$ $I_\theta$-braidable, where $\ch$ is the GNS Hilbert space of the canonical trace of $\mathbb{A}_\theta^\bn$ , $\xi$ its GNS cyclic vector, and $\iota_j$ the composition of the embedding $i_j$ of $C(\bt)$ into $\mathbb{A}_\theta^\bn$ with the GNS representation itself.\\
This can be seen as follows. Let us denote by $u_j$ is the image of $f_0(z)=z$, $z$ in $\bt$, under the $j$-th embedding $i_j$. First of all,
arguing as in the part of the proof of Theorem \ref{counterexamplebraid} where we addressed the Fourier expansion of
$\sigma_1(u_1)\xi$ and taking into account that the GNS vector $\xi$ is separating for the von Neumann algebra generated  in the GNS representation of the trace, we  see that $\sigma_1(u_1)$ equals $u_0^*u_1^2$ up to a phase, say $\lambda$. From the braid relations it then follows that $\sigma_i(u_{i-1})=u_i$, $\sigma_i(u_i)=\lambda u_{i-1}^*u_i^2$ for all $i$ in $\bn$, and  $\sigma_i(u_j)= u_j$ for $j\neq i, i-1$. Finally, the phase $\lambda$ is determined by 
imposing the relation $\sigma_i \sigma_{i+1}\sigma_i = \sigma_{i+1} \sigma_i \sigma_{i+1}$, which can only be fulfilled if $\lambda=e^{4\pi i\theta}$.
\end{rem}

\section{Extended de Finetti's theorem for processes whose distribution factorizes through infinite twisted tensor products}

This section is devoted to presenting a version of de Finetti's theorem
for processes whose distribution factorizes through a twisted tensor product. In some of the  models  we discuss,
permutations may not be implemented as automorphisms on the corresponding quotients.
Even so, the compact convex set of all spreadable states can be studied and determined completely.
In particular, we aim to understand when this set is a Choquet simplex and what its extreme states are.
To begin with, when permutations do act on a twisted product, {\it i.e.} when the assigned bicharacter is
antisymmetric, spreadable states and exchangeable states are the same, as follows by
Theorem \ref{mainRyll}.
When permutations do not act,
there are going to be more spreadable states than exchangeable states, and, as we saw in Corollary \ref{Rylltwisted},  models exist where the set of 
exchangeable states is even empty.\\
Studying spreadability entails having to do with $\bj_\bn$, which  is only a semigroup. Therefore,
the general theory of $C^*$-dynamical systems is not directly applicable.
However, it is still possible to overcome this obstacle by resorting to an automorphic extension of the dynamics, which is a procedure 
first introduced in \cite{CDRCAR}.\\
As of now $(\ga, G, \gamma)$ will be a $G$-graded unital $C^*$-algebra. Moreover, henceforward we will always be working under the assumption that   
$\ga^G$ is nuclear.
\begin{rem}\label{separate}
We would like to point out that product states separate $\otimes_v^\bn\ga$ if $\ga^G$ is nuclear. Indeed, the $C^*$-subalgebra $C^*(\{i_j(\ga^G): j\in\bn\})\subset\otimes_v^\bn\ga $ identifies with
the infinite tensor product $\otimes^\bn \ga^G$, without having to specify what completion we are considering as they are all the same, and in particular we can assume we are working with the spatial completion.  Therefore, the set $\{\otimes^\bn\om: \om\in\cs(\ga^G)\}$ separates
 $\otimes^\bn \ga^G$. Now for  any $G$-invariant state $\om$ on $\ga$, we have $\times \om= \otimes\om\upharpoonright_{\ga^G}\,\circ E$, where $E$ is the canonical conditional 
expectation of  $\otimes_v^\bn\ga$ onto $\otimes^\bn \ga^G$. The thesis then follows because $E$ is faithful.
\end{rem}
We will need to use the above remark in the proof of Lemma \ref{subuniversal}.\\
Let $\varphi$ be a spreadable state on $\bigotimes_v^\bn\ga$. Associated with any $h\in\bj_\bn$ there is an isometry acting on the Hilbert space
$\ch_\varphi$ of the GNS representation of $\varphi$ as
$$T_h^\varphi\pi_\varphi(x)\xi_\varphi:=\pi_\varphi(\a_h(x))\xi_\varphi\,, \,\, x\in \bigotimes_v^\bn\ga\,,$$
where
$\xi_\varphi$ is the GNS vector of $\varphi$.\\
Let $\bz\big[\frac{1}{2}\big]$ be the  (additive) group of dyadic numbers. For every fixed $n\in\bn$, we denote
by $\frac{\bz}{2^n}$ the set of rational numbers $\{\frac{k}{2^n}: k\in\bz\}$.\\
In order to obtain the automorphic extension of the dynamics, we need to think of the generators of
$\bj_\bz$, see the preliminary section, as functions acting on the whole set of dyadic numbers. We do so by suitably extending $\theta_0$ to
a bijection $\widetilde{\theta}_0$ of $\bz\big[\frac{1}{2}\big]$, which is defined below

$$
\widetilde{\theta}_0(d):=\left\{\begin{array}{ll}
                      d & \text{if}\,\, d\leq -1\,, \\
                      2d+1&\text{if}\,\, -1\leq d\leq 0\\
                      d+1 & \text{if}\,\, d\geq 0\,.
                    \end{array}
                    \right.
$$

 For each natural $n$, we consider the dilation $\delta_n$ by $2^n$ acting on $\bz\big[\frac{1}{2}\big]$, that is
$$\delta_n(d)=2^n d\,, \,\, d\in \bz\bigg[\frac{1}{2}\bigg]\,. $$
Let us define  $\widetilde{\theta}_n:=\delta_n^{-1}\circ\widetilde{\theta}_0\circ\delta_n$ and
$\widetilde{\tau}_{k, n}(r)=r+\frac{k}{2^n}\,, \, r\in \bz\big[\frac{1}{2}\big]$, for $n\in\bn$ and $k\in\bz$.\\
For each natural $n$, we then consider the group $H_n$ generated by
$\widetilde{\theta}_k$ and $\widetilde{\tau}_{1, n}$ with $k=1, \ldots, n$. 
\begin{rem}\label{Thompson}
We point out that each
$H_n$ is isomorphic with the Thompson group $F$ by \cite[Proposition 3.1.1]{Belk}.
\end{rem}
Note that $H_n\subset H_{n+1}$ for each $n$. Therefore, we can consider the group $H$ which is by definition the inductive  limit
of the sequence $\{H_n: n\in\bn\}$ w.r.t. the inclusions, namely
$$H:= \bigcup_{n=0}^\infty H_n\, .$$\\
Notably, $H$  acts through automorphisms on $\bigotimes_v^{\bz\left[\frac{1}{2}\right] }\ga$,
the twisted infinite tensor product of $\ga$ with itself indexed by $\bz\left[\frac{1}{2}\right]$,
 by moving the corresponding indices. Precisely,  for every $h\in H$, by universality of $\bigotimes_v^{\bz\left[\frac{1}{2}\right] }\ga$ there exists
a  $*$-automorphism $\alpha_h$ of $\bigotimes_v^{\bz\left[\frac{1}{2}\right] }\ga$ uniquely determined by
$$\alpha_h(i_d(a)):=i_{h(d)}(a) \,\,\, d\in\bz\left[\frac{1}{2}\right]\,, a\in\ga\,.$$
We denote by $\cs^H\big(\bigotimes_v^{\bz\left[\frac{1}{2}\right] }\ga\big)$ the set of all invariant states on
$\bigotimes_v^{\bz\left[\frac{1}{2}\right] }\ga$ under the action of $H$.\\

The next proposition allows us to realize any spreadable state on $\bigotimes_v^\bn\ga$ as a unique
$H$-invariant state on $\bigotimes_v^{\bz\left[\frac{1}{2}\right] }\ga$.
\begin{prop}
\label{affine}
The restriction map $T: \cs^H\big(\bigotimes_v^{\bz\left[\frac{1}{2}\right] }\ga\big)\rightarrow \cs^{\bj_\bn}(\bigotimes_v^\bn\ga)$
$$T(\varphi)=\varphi\lceil_{\bigotimes_v^\bn\ga}$$
establishes an affine homeomorphism of compact convex sets.
\end{prop}
\begin{proof}
It is pretty much the same as the proof
of \cite[Proposition 3.1]{CDRCAR}. There are only two things to care about in the present context. 
The first is that
any $\bj_\bn$-invariant state on $\bigotimes_v^\bn\ga$ uniquely extends to a $\bj_\bz$-invariant state on
 $\bigotimes_v^\bz\ga$. Indeed, the restriction map $\cs(\bigotimes_v^\bz\ga)^{\bj_\bz}\ni\varphi\mapsto \varphi\upharpoonright_{\bigotimes_v^\bn\ga}\in\cs(\bigotimes_v^\bn\ga)^{\bj_\bn}$ is clearly injective. As for
its surjectivity, starting from a $\bj_\bn$-invariant state $\varphi_0$ on $\bigotimes_v^\bn\ga$, one first defines a sequence of states $\{\varphi_n: n\in\bn\}$, where $\varphi_n$ is a state on the infinite twisted product of $\ga$ with itself indexed by $\{-n,\ldots, 0, 1, \ldots\}$,  $\bigotimes_{v\,\,{k=-n}}^\infty\ga_k$, 
$\ga_k=\ga$, as $\varphi_n:=\varphi_0\circ\tau^{n}$.  Now the sought $\bj_\bz$-invariant extension
of $\varphi_0$ is the unique state $\widetilde{\varphi}$ on $\bigotimes_v^\bz\ga$ such that
the restriction of $\varphi$ to $\bigotimes_{v\,\,{k=-n}}^\infty\ga_k$ is $\varphi_n$ for every $n$.\\
The second thing to care about is that,  for each $n$, the $C^*$-subalgebra of  $\bigotimes_v^{\bz\left[\frac{1}{2}\right] }\ga$
generated by the range of the embeddings $i_j$ for $j$ in $\frac{\bz}{2^n}$
is still universal, being $\bigotimes_v^{\frac{\bz}{2^n }}\ga$, as one would expect. This is proved in Lemma \ref{subuniversal}.
\end{proof}
\begin{lem}\label{subuniversal}
For each $n$, the $C^*$-subalgebra of $\bigotimes_v^{\bz\left[\frac{1}{2}\right] }\ga$
generated by $\{i_j(\ga): j\in\frac{\bz}{2^n } \}$ is canonically $*$-isomorphic
with $\bigotimes_v^{\frac{\bz}{2^n }}\ga$.
\end{lem}
\begin{proof}
By universality of $\bigotimes_v^{\frac{\bz}{2^n }}\ga$, there certainly is a $*$-epimorphism
$\Phi: \bigotimes_v^{\frac{\bz}{2^n }}\ga\rightarrow C^*(\{i_j(\ga): j\in\frac{\bz}{2^n } \})$ such that
\begin{equation}\label{epi}
\Phi(i_j'(a))=i_j(a)
\end{equation} for every $j$ in ${\frac{\bz}{2^n }}$, and every $a$ in $\ga$, where
$\{i_j': j\in{\frac{\bz}{2^n }} \}$ are the embeddings of $\ga$ into $ \bigotimes_v^{\frac{\bz}{2^n }}\ga$
appearing in \eqref{commrules}. \\
%
For every $G$-invariant state $\om$ on $\ga$, 
we denote by $\times\om$ the corresponding infinite product state on $ \bigotimes_v^{\frac{\bz}{2^n }}\ga$, and
by $\varphi_\om$ the restriction to $C^*(\{i_j(\ga): j\in\frac{\bz}{2^n } \})$ of the infinite product of $\om$ with itself as a state of $\bigotimes_v^{\bz\left[\frac{1}{2}\right] }\ga$.
Thanks to \eqref{epi}, we clearly have $\times\om=\varphi_\om\circ\Phi$. The conclusion that $\Phi$ is injective can now be easily drawn: if $x$ in  $\bigotimes_v^{\frac{\bz}{2^n }}\ga$ is such that $\Phi(x^*x)=0$, then
$\times\om\, (x^*x)=0$ for every $G$-invariant state $\om$ on $\ga$, hence $x^*x=0$ because product states separate 
$\bigotimes_v^{\frac{\bz}{2^n }}\ga$, see Remark \ref{separate}.
\end{proof}
\begin{rem}\label{representations}
We take the opportunity to point out that an analogue of Proposition \ref{affine} holds for the free product
$\ast_\bn\ga\subset \ast_{\bz\left[\frac{1}{2}\right]} \ga$ as well. This allows one to associate a unitary
representation of the Thompson group $F$ to any spreadable quantum stochastic process thanks to the group isomorphism recalled in
Remark \ref{Thompson}. This is obtained by $H_0\ni h\mapsto u_h^\varphi$ since $H_0\cong F$. 
We note that the representation of $F$ coming from $H_n\cong F$
is unitarily  equivalent to that coming from $H_0$, 
since the two are intertwined by  the unitary implementing $\delta_n$.
\end{rem}
In light of the affine homeomorphism established in Proposition \ref{affine}, studying the properties
of the convex set of spreadable states amounts to studying $\cs^H\big(\bigotimes_v^{\bz\left[\frac{1}{2}\right] }\ga\big)$. With this in mind, it is then natural to ask oneself when the extended system is $H$-abelian, which is the same as saying  $\cs^H\big(\bigotimes_v^{\bz\left[\frac{1}{2}\right] }\ga\big)$ is a Choquet simplex by Lemma \ref{summary}.\\
For any given $H$-invariant state $\varphi$ on $\bigotimes_v^{\bz\left[\frac{1}{2}\right] }\ga$ and for every $h\in G$, we can define
a unitary $u_h^\varphi$ acting on the GNS Hilbert space $\ch_\varphi$  of the GNS representation $\pi_\varphi$ of  $\bigotimes_v^{\bz\left[\frac{1}{2}\right] }\ga$  as
$$u_h^\varphi\pi_\varphi(a)\xi_\varphi=\pi_\varphi(\alpha_h(a))\xi_\varphi\,,\, \, a\in\bigotimes_v^{\bz\left[\frac{1}{2}\right] }\ga\,,$$
where $\xi_\varphi$ is the GNS vector.\\
We then consider
the closed subspace
$$\ch_\varphi^H:=\{\xi\in\ch_\varphi: u_h^\varphi\xi=\xi,\, \textrm{for all}\, h\in H\}$$
of invariant
vectors, and denote by $E_\varphi$ the corresponding orthogonal projection.
Finally, we denote by $\varphi_0$ the restriction of $\varphi$ to $\bigotimes_v^\bn\ga$.\\
We next need to consider the unital semigroup $\widetilde{S}\subset H$ generated by
$\widetilde{\theta_n}$ and $\widetilde{\tau}_{k, n}$ as $n$ varies in $\bz$ and $k$ in $\bn$. Note that for any
element $h\in \widetilde{S}$ one has $h(q)\geq q$ for all $q$ in $\bz\big[\frac{1}{2}\big]$.\\
Now, for any $H$-invariant state $\varphi$ on $\bigotimes_v^{\bz\left[\frac{1}{2}\right] }\ga$ define
$\ch_\varphi^{\widetilde{S}}:=\{\xi\in\ch_\varphi: u_h^\varphi\xi=\xi,\, h\in \widetilde{S}\}$.
Since $\widetilde{S}$ generates $H$ as a group, we clearly have

\begin{lem}\label{fromGtoS}
For any $H$-invariant state $\varphi$ on $\bigotimes_v^{\bz\left[\frac{1}{2}\right] }\ga$, one has $\ch_\varphi^H=\ch_\varphi^{\widetilde{S}}$.
\end{lem}

The following result is key to studying $H$-abelianness of the extended system.
\begin{lem}\label{commrulproj}
Let $\varphi$ be an $H$-invariant state on $\bigotimes_v^{\bz\left[\frac{1}{2}\right] }\ga$.
For homogeneous $a, b$ in $\bigotimes_v^{\bz\left[\frac{1}{2}\right] }\ga$, the commutation rule
$$E_\varphi \pi_\varphi(a)E_\varphi \pi_\varphi(b)E_\varphi=v(\partial a, \partial b)E_\varphi \pi_\varphi(b)E_\varphi \pi_\varphi(a)E_\varphi\,$$
holds.
\end{lem}

\begin{proof}
Define the twisted commutator for homogeneous elements $a, b$ as
$$ [a, b]_{\partial a, \partial b}:=ab-v (\partial a, \partial b)ba, \, a, b\in \bigotimes_v^{\bz\left[\frac{1}{2}\right] }\ga\, .$$
The first thing we show is that the limit equality
\begin{equation}\label{twistesasab}
\lim_{n\rightarrow\infty} \| [ \a_{\widetilde{\tau}_{n, 0}}(a), b]_{\partial a, \partial b}\|=0 
\end{equation}
holds for any homogeneous $a, b$. Indeed, the equality is easily seen to hold for localized elements, since we even have
$[ \a_{\widetilde{\tau}_{n, 0}}(a), b]_{\partial a, \partial b}=0$ provided  $n$ is so big that 
the supports of  $\a_{\widetilde{\tau}_{n, 0}}(a)$ and $b$ are disjoint. The general case can be dealt with by
a density argument because any  homogeneous $a$ is the norm limit of a sequence $\{a_n\}_{n\in\bn}$, where each $a_n$ is localized and homogeneous with $\partial a_n=\partial a$. This can be seen in the following way.
Since the $*$-subalgebra of localized elements is dense, there certainly exists a sequence
$\{x_n\}_{n\in\bn}$ of localized elements such that $\|a-x_n\|\rightarrow 0$ for $n\rightarrow\infty$.
For every $n$ in $\bn$, define $a_n :=\int_G \gamma_g(x_n)\overline{\partial a(g)} {\rm d}\mu_G(g)$.
Now $a_n$ is homogeneous of the same degree as $a$, for 
\begin{align*}
\gamma_h(a_n)&=\int_G \gamma_{hg}(x_n)\overline{\partial a(g)} {\rm d}\mu_G(g)=\int_G \gamma_g(x_n)\overline{\partial a(h^{-1}g)} {\rm d}\mu_G(g)\\
&=\partial a(h)a_n\,, \,\,\,\textrm{ for all}\,\, h\in G\,.
\end{align*}
Furthermore, the sequence $\{a_n\}_{n\in\bn}$ still converges to $a$ as
\begin{align*}
\lim_n a_n=&\lim_n \int_G \gamma_g(x_n)\overline{\partial a(g)}{\rm d}\mu_G(g)=\int_G \gamma_g(a)\overline{\partial a(g)}{\rm d}\mu_G(g)\\
=&\int_G \partial a(g)\overline{\partial a(g)}{\rm d}\mu_G(g)\,\,a=a\,,
\end{align*}
where the second-to-last equality holds by uniform convergence.\\

Now, by applying the Alaoglu-Birkhoff ergodic theorem to the family of unitaries $\{u_h^\varphi: h\in \widetilde{S}\}$, we find that $E_\varphi$ is the strong limit of
a net whose elements are finite convex combination of the type
$\sum_{i=1}^m\lambda_i u_{h_i}$, with $h_i$ in $\widetilde{S}$, $\lambda_i\geq 0$ for all $i=1, \ldots, m$, and
$\sum_{i=1}^m \lambda_i =1$.\\
Thanks to the twisted asymptotic abelianness in \eqref{twistesasab}, for any homogeneous $a, b$, for any
$\eta, \xi$ in ${\rm Ran} E_\varphi$ with $\|\eta\|=\|\xi\|=1$, and
any $\varepsilon >0$ we have
\begin{equation}
\left|   \langle \eta, \left[\pi_\varphi\left(\a_{\widetilde{\tau}_{n, 0}}  \left( \sum_{i=1}^m\lambda_i\a_{h_i}(a)\right)\right), \pi_\varphi(b)\right]_{\partial a, \partial b}\xi \rangle    \right|\leq \varepsilon
\end{equation}
for $n\geq N_\varepsilon$.
We have the following chain of inequalities:
\begin{align*}
&\varepsilon\geq\left|   \langle \eta, \left[\pi_\varphi\left(\a_{\widetilde{\tau}_{n, 0}}  \left( \sum_{i=1}^m\lambda_i\a_{h_i}(a)\right)\right), \pi_\varphi(b)\right]_{\partial a, \partial b}\xi \rangle    \right|=\\
&|  \langle\eta, \sum_{i=1}^m \lambda_i u_{\widetilde{\tau}_{n, 0}}^\varphi u_{h_i}^\varphi\pi_\varphi(a)(u_{h_i}^\varphi)^* (u_{\widetilde{\tau}_{n, 0}}^\varphi)^*\pi_\varphi(b)\xi\rangle -\\
&v(\partial a, \partial b) \langle\eta, \sum_{i=1}^m \lambda_i \pi_\varphi(b)u_{\widetilde{\tau}_{n, 0}}^\varphi u_{h_i}^\varphi\pi_\varphi(a)(u_{h_i}^\varphi)^* (u_{\widetilde{\tau}_{n, 0}}^\varphi)^*)\xi\rangle |=\\
&|  \langle \sum_{i=1}^m\lambda_iu_{h_i}^\varphi\pi_\varphi(a^*)\eta,   (u_{\widetilde{\tau}_{n, 0}}^\varphi)^*\pi_\varphi(b)\xi\rangle -v(\partial a, \partial b) \langle\eta, \pi_\varphi(b)u_{\widetilde{\tau}_{n, 0}}^\varphi \sum_{i=1}^m \lambda_i u_{h_i}^\varphi\pi_\varphi(a))\xi\rangle|\geq\\
&|  \langle E_\varphi\pi_\varphi(a^*)\eta,   \pi_\varphi(b)\xi\rangle -v(\partial a, \partial b) \langle\pi_\varphi(b^*)\eta,  E_\varphi\pi_\varphi(a))\xi\rangle|-\\
&|  \langle( \sum_{i=1}^m\lambda_iu_{h_i}^\varphi-E_\varphi)\pi_\varphi(a^*)\eta,   (u_{\widetilde{\tau}_{n, 0}}^\varphi)^*\pi_\varphi(b)\xi\rangle -v(\partial a, \partial b) \langle\eta, \pi_\varphi(b)u_{\widetilde{\tau}_{n, 0}}^\varphi( \sum_{i=1}^m \lambda_i u_{h_i}^\varphi-E_\varphi)\pi_\varphi(a))\xi\rangle|
\end{align*}
Now by the mentioned Alaoglu-Birkhoff theorem we can also assume
$$\left\|\left(\sum_{i=1}^m\lambda_iu_{h_i}^\varphi-E_\varphi\right)\pi_\varphi(a^*)\eta\right\|\leq\varepsilon\,\,{\rm and}\,\, \left\|\left(\sum_{i=1}^m \lambda_i u_{h_i}^\varphi-E_\varphi\right)\pi_\varphi(a))\xi\right\|\leq\varepsilon\,.$$
But then
\begin{align*}
&|  \langle( \sum_{i=1}^m\lambda_iu_{h_i}^\varphi-E_\varphi)\pi_\varphi(a^*)\eta,   (u_{\widetilde{\tau}_{n, 0}}^\varphi)^*\pi_\varphi(b)\xi\rangle -v(\partial a, \partial b) \langle\eta, \pi_\varphi(b)u_{\widetilde{\tau}_{n, 0}}^\varphi(\sum_{i=1}^m \lambda_i u_{h_i}^\varphi-E_\varphi)\pi_\varphi(a)\xi\rangle|\\
&\geq |  \langle E_\varphi\pi_\varphi(a^*)\eta,   \pi_\varphi(b)\xi\rangle -v(\partial a, \partial b) \langle\pi_\varphi(b^*)\eta,  E_\varphi\pi_\varphi(a))\xi\rangle|-\varepsilon\|\pi_\varphi(b)\xi\|-\varepsilon\|\pi_\varphi(b^*)\eta\|\, .
\end{align*}
Since $\varepsilon>0$ can be chosen as small as wished, we must have 
$$  \langle E_\varphi\pi_\varphi(a^*)\eta,   \pi_\varphi(b)\xi\rangle-v(\partial a, \partial b) \langle\pi_\varphi(b^*)\eta,  E_\varphi\pi_\varphi(a))\xi\rangle=0\,,$$
from which the statement easily follows.
\end{proof}
The next result provides sufficient conditions for $H$-abelianness to hold.
\begin{thm}\label{suffcondabelianness}
The $C^*$-dynamical system $(\bigotimes_v^{\bz\left[\frac{1}{2}\right] }\ga, H, \a)$ is
$H$-abelian if: for every $\chi, \eta$ in $\Delta_v$ either of the following conditions holds
\begin{enumerate}
\item [(i)] $v(\chi, \eta)=1$
\item [(ii)] $v_S(\chi, \eta)\neq 1$\,.
\end{enumerate}
\end{thm}
\begin{proof}
We have to show that, under our hypotheses,  for any $H$-invariant state $\varphi$ the set $E_\varphi \pi_\varphi(\bigotimes_v^{\bz\left[\frac{1}{2}\right] }\ga) E_\varphi$ is commutative.\\
The proof will be done by analyzing three cases.
We first deal with elements $a$ belonging to some spectral eigenspace with $\partial a\in\widehat{G}\setminus\Delta_v$.\\
We show that
\begin{equation}\label{oddcomponent}
E_\varphi\pi_\varphi(a)E_\varphi=0,\,\,\, \textrm{for all}\,\, a\,\,{\rm with}\,\, \partial a\in\widehat{G}\setminus\Delta_v\,.
\end{equation} 
This is a straightforward  application of Lemma \ref{commrulproj}, which yields
$$E_\varphi\pi_\varphi(a )E_\varphi\pi_\varphi(a^*)E_\varphi=v(\partial a, -\partial a)E_\varphi\pi_\varphi(a^* )E_\varphi\pi_\varphi(a)F_\varphi\, .$$
Now, since $v(\partial a, -\partial a)$ is a non-trivial phase, the spectrum of $E_\varphi\pi_\varphi(a )E_\varphi\pi_\varphi(a^*)E_\varphi$ is $\{0\}$, hence $E_\varphi\pi_\varphi(a )E_\varphi\pi_\varphi(a^*)E_\varphi=0$.\\
The second  case is when $a, b$ are homogeneous with $\partial a, \partial b$ in $\Delta_v$ such that
$v(\partial a, \partial b)=1$, in which case Lemma \ref{commrulproj} says that
$E_\varphi \pi_\varphi(a)E_\varphi$ and $E_\varphi \pi_\varphi(b)E_\varphi$ commute.\\
The third case  is when $a, b$ are homogeneous with $\partial a, \partial b$ in $\Delta_v$ and
$v_s(\partial a, \partial b)\neq 1$. We have

\begin{align*}
E_\varphi \pi_\varphi(a)E_\varphi \pi_\varphi(b)E_\varphi&=v(\partial a, \partial b)E_\varphi \pi_\varphi(b)E_\varphi \pi_\varphi(a)E_\varphi\\
&=v(\partial a, \partial b) v(\partial b, \partial a)E_\varphi \pi_\varphi(a)E_\varphi \pi_\varphi(b)E_\varphi\\
&=v_s(\partial a, \partial b)E_\varphi \pi_\varphi(a)E_\varphi \pi_\varphi(b)E_\varphi\,,
\end{align*}
hence $E_\varphi \pi_\varphi(a)E_\varphi \pi_\varphi(b)E_\varphi$ must be zero since $v_s(\partial a, \partial b)\neq 1$.
\end{proof}
Although we will only make use of the sufficient conditions given in the above theorem, $H$-abelianness can also be  characterized by means of a necessary and sufficient condition.
\begin{thm}\label{charHabelianness}
For the $C^*$-dynamical system $(\bigotimes_v^{\bz\left[\frac{1}{2}\right] }\ga, H, \a)$ not to be
$H$-abelian, it is necessary and sufficient that there exists a $\mathbb{J}_\bn$-invariant
state $\varphi$ on $\bigotimes_v^\bn\ga$ such that $\varphi(ab)\neq 0$ for some
localized homogeneous $a, b$ with disjoint supports and the support of one to left
of the support of the other, $\partial a, \partial b\in \Delta_v$ and
$v(\partial a, \partial b)\neq 1$.\\
\end{thm}
\begin{proof}
We start by proving that if there exists a state $\varphi$ with the properties in the statement, then the system
fails to be $H$-abelian.  For $a, b$ as in statement, thanks to
Lemma \ref{commrulproj} we have that
$E_\varphi \pi_\varphi(a)E_\varphi$ and $E_\varphi \pi_\varphi(b)E_\varphi$ do not commute provided that $E_\varphi \pi_\varphi(a)E_\varphi\pi_\varphi(b)E_\varphi\neq 0$.
Therefore, we only need to make sure $E_\varphi \pi_\varphi(a)E_\varphi\pi_\varphi(b)E_\varphi$ is different from zero.\\
 By the Alaoglu-Birkhoff ergodic theorem, see \cite[Proposition 4.3.4]{BR1} there exists a net indexed by $\beta\in I$ such that
$E_\varphi=\lim_\beta \sum_{i=1}^{n_\beta}\lambda_i^\beta u_{h_i^\beta}^\varphi $ (in the strong topology)
where, for all indices $\beta$,  the finite set $\{h_i^\beta: i=1, \ldots, n_\beta\}$ lies in  $\widetilde{S}$ ({\rm i.e. } $h_i^\beta(r)\geq r$ for all dyadic $r$'s), $\lambda_i^\beta\geq 0$ with $\sum_{i=1}^{n_\beta}\lambda_i^\beta=1$. 
Then we have
\begin{align*}
&\langle E_\varphi \pi_\varphi(a)E_\varphi\pi_\varphi(b)E_\varphi\xi_\varphi, \xi_\varphi\rangle=
\langle \pi_\varphi(a)E_\varphi\pi_\varphi(b)\xi_\varphi, \xi_\varphi\rangle=\\
&\lim_\beta \langle \pi_\varphi(a)\sum_{i=1}^{n_\beta}\lambda_i^\beta u_{h_i^\beta}^\varphi\pi_\varphi(b)\xi_\varphi, \xi_\varphi\rangle= \lim_\beta\sum_{i=1}^{n_\beta} \lambda_i^\beta \langle \pi_\varphi(a) u_{h_i^\beta}^\varphi\pi_\varphi(b)\xi_\varphi, \xi_\varphi\rangle\,.
\end{align*}
We next show that, for all $\beta$'s,  $\langle \pi_\varphi(a) u_{h_i^\beta}^\varphi\pi_\varphi(b)\xi_\varphi, \xi_\varphi\rangle=\varphi(ab)$ for every $i=1, \ldots, n_\beta$.
Without any loss of generality, we can assume that the support of $a$ sits to the left of that of $b$.
We have
\begin{align*}
\langle \pi_\varphi(a) u_{h_i^\beta}^\varphi\pi_\varphi(b)\xi_\varphi, \xi_\varphi\rangle&=
\langle \pi_\varphi(a)\pi_\varphi(\a_{h_i }^\beta (b))\xi_\varphi, \xi_\varphi\rangle=
\langle \pi_\varphi(a \a_{h_i }^\beta (b))\xi_\varphi, \xi_\varphi\rangle\\
&=\varphi(ab)\,,
\end{align*}
where the last equality is got to by $H$-invariance (considering a map that fixes the support of $a$ pointwise and acts as $h_i^\beta$ on the support of $b$). 
In particular, we have shown the equality
$$\langle E_\varphi \pi_\varphi(a)E_\varphi\pi_\varphi(b)E_\varphi\xi_\varphi, \xi_\varphi\rangle=\varphi(ab)\,,$$
from which it follows at once that $ E_\varphi \pi_\varphi(a)E_\varphi\pi_\varphi(b)E_\varphi$ cannot be zero.\\
For the inverse implication, suppose the system fails to be $H$-abelian. Then there exists an $H$-invariant state
$\varphi$ and $a, b$ in $\bigotimes_v^{\bz\left[\frac{1}{2}\right] }\ga$ such that
$ E_\varphi \pi_\varphi(a)E_\varphi$ and $ E_\varphi \pi_\varphi(b)E_\varphi$
do not commute.\\ 
Now the same density argument employed in the proof of Lemma \ref{commrulproj}
says that the linear span of homogeneous localized elements is a norm dense $*$-subalgebra of the whole
$\bigotimes_v^{\bz\left[\frac{1}{2}\right] }\ga$. Therefore, we may suppose
that $a, b$ are both homogeneous and localized.
In addition, without any lack of generality occurring, we can take $a, b$ such that the support of $a$ lies entirely to the left of the support of $b$.
By Lemma \ref{commrulproj}, for  $ E_\varphi \pi_\varphi(a)E_\varphi$ and $ E_\varphi \pi_\varphi(b)E_\varphi$ not to commute, it is necessary
$\partial a, \partial b$ are such that $v(\partial a, \partial b)\neq 1$ and $ E_\varphi \pi_\varphi(a)E_\varphi\pi_\varphi(b)E_\varphi\neq 0$.
In particular, there exist $\xi, \eta$ in ${\rm Ran} E_\varphi$ such that
 $\langle E_\varphi \pi_\varphi(a)E_\varphi\pi_\varphi(b)E_\varphi \xi,\eta \rangle=\langle  \pi_\varphi(a)E_\varphi\pi_\varphi(b) \xi,\eta\rangle \neq 0$.\\
By the same reasoning as above we also have
\begin{equation*} 
\langle  \pi_\varphi(a)\pi_\varphi(b) \xi,\eta\rangle=\langle  \pi_\varphi(a)E_\varphi\pi_\varphi(b) \xi,\eta\rangle\neq 0\,.
\end{equation*}
Considering the Cartesian decomposition of  the $H$-invariant functional $\langle \cdot\xi, \eta\rangle$, we find that there also exists
a $H$-invariant hermitian linear funcional $\om$ such $\om(ab)\neq 0$. 
Now $\om$ uniquely decomposes as the difference $\om=\om_+- \om_-$, where $\om_{\pm}$ are positive linear functionals.
The uniqueness of such decomposition (the Hahn-Jordan decomposition) implies that $\om_{\pm}$ are still $H$-invariant, and at least one of the two satisfies the thesis.
\end{proof}

We have now endowed ourselves with all the necessary tools to completely determine the Choquet simplex of spreadable states. 
\begin{thm}\label{maindefinetti}
Let $v:\widehat{G}\times\widehat{G}\rightarrow\bt$ be a bicharacter such that either $v(\chi, \eta)=1$
or $v_S(\chi, \eta):=v(\chi, \eta)v(\eta, \chi)\neq 1$ for $\chi, \eta$ in $\Delta_v$.\\
The convex compact set $\cs^{\bj_\bn}(\bigotimes_v^\bn\ga)$ is the Bauer
simplex whose extreme set coincides with the set of all product states $\{\times\om: \om\in\cs_v(\ga)\}$.\\

\end{thm}

\begin{proof}
We start by observing that, under the hypotheses on $v$,  Theorem \ref{suffcondabelianness} guarantees that the extended system
is $H$-abelian, and thus $\cs^{\bj_\bn}(\bigotimes_v^\bn\ga)$ is a Choquet simplex by Proposition \ref{affine}.\\
We next show that any extreme state $\varphi$ of  $\cs^{\bj_\bn}(\bigotimes_v^\bn\ga)$  is an infinite product state for some
$\om$ in $\cs(\ga)$. This can be seen by showing that the restriction of any extreme state in $\cs^H\big(\bigotimes_v^{\bz\left[\frac{1}{2}\right]})$ to $\bigotimes_v^\bn\ga$
is an infinite product state.\\
If $\varphi$ is such a state, then
$E_\varphi$ is the rank-one projection onto $\bc\xi_\varphi$. 
The first thing to do is to apply the Alaloglu-Birkhoff theorem to the family  of isometries
$\{T_h^\varphi: h\in \widetilde{S}\}$. This
yields a net $\{S_\gamma: \gamma\in I\}$, whose terms are finite convex combinations of the form $S_\gamma=\sum_{k=1}^{n_\gamma} \lambda_{k}^\gamma  T_{h_k^\gamma}^\varphi$, for some $\lambda_k^\gamma\geq 0$ with $\sum_{k=1}^{n_\gamma} \lambda_k^\gamma=1$, which is
strongly convergent to  $E_\varphi^S$. 
Let now $a, b$ be fixed elements of $\bigotimes_v^\bn\ga$. For each $\gamma$ in $I$, define $b_\gamma:=\sum_{k=1}^{n_\gamma} \lambda_{k}^\gamma  \a_{h_k^\gamma}(b)$. We have
\begin{align}
\begin{split}
\label{clustering}
\lim_\g \varphi(ab_\g)&= \lim_\g \langle \pi_\varphi(a b_\g) \xi_\varphi, \xi_\varphi \rangle\\
&= \lim_\g  \langle  \pi_\varphi(a)  \sum_{k=1}^{n_\g}\lambda_{k}^\g T_{h_k^\g}^\varphi \pi_\om(b)\xi_\varphi, \xi_\varphi\rangle
\\
&=\langle  \pi_\varphi(a)   E_{\bc\xi_\varphi}\pi_\varphi(b)\xi_\varphi, \xi_\varphi\rangle=
\varphi(a)\varphi(b)\,.
\end{split} .
\end{align}
Let $\om$ be  state on $\ga$ defined by
$\om(a):=\varphi(i_l(a))$, $a$ in $\ga$, where $i_l$ is any of the embeddings
of $\ga$ into  $\bigotimes_v^\bn\ga$. 
Note that $\om$ is well
defined, {\it i.e.} the definition does not depend on $l$, because $\varphi$ is spreadable.
We aim to show the equality
$$\varphi(i_{j_1}(a_1)i_{j_2}(a_2)\cdots i_{j_n}(a_n))=\om(a_1)\om(a_2)\cdots\om(a_n)$$
for every $n\in\bn$, for every $j_1< j_2< \ldots< j_n\in\bn$, and for every $a_1, a_2, \ldots, a_n$ in $\ga$.
This task can be accomplished by induction on $n$. For $n=1$, there is nothing to prove.
Let us move on to take the inductive step. We will show that the equality holds with $n+1$ if
it holds with $n$. To this end, we start by observing that, for all $n\in\bn$, $j_1<j_2<\ldots<j_{n+1}$ in $\bn$, and $a_1, \ldots, a_n, b\in\ga$, $H$-invariance of $\varphi$
gives
$$\varphi(i_{j_1}(a_1)\cdots i_{j_n}(a_n)i_{j_{n+1}}(b))=\varphi(i_{j_1}(a_1)\cdots i_{j_n}(a_n) i_{h_k^\g(j_{n+1})}(b))\,,$$
where, for each $\g$ in $I$, the $h_k^\g$'s are the monotone functions in $\widetilde{S}$ (in particular, $h_k^\g(m)\geq m$ for all $m$ in $\bn$) appearing in the definition of the net $S_\g$ we introduced above. By summing the equalities above on all $k$'s between $1$ and $n_\g$, we find
\begin{align*}
&\varphi(i_{j_1}(a_1)\cdots i_{ j_n}(a_n)i_{j_{n+1}}(b))\\
=&\sum_{k=1}^{n_\g}\lambda_k^\g\varphi(i_{j_1}(a_1)\cdots i_{j_n}(a_n) i_{h_k^\g(j_{n+1})}(b))\\=&\varphi(i_{j_1}(a_1)\cdots i_{j_n}(a_n)b_\g)
\end{align*}
\medskip
where $b_\g:=\sum_{k=1}^{n_\g}\lambda_k^\g i_{h_k^\g(j_{n+1})}(b)=\sum_{k=1}^{n_\g} \lambda_{k}^\g \a_{h_k^\g}(i_{j_{n+1}}(b))$.\\
\medskip
Now thanks to \ref{clustering} we have
$$\lim_\g \varphi(i_{j_1}(a_1)\cdots i_{j_n}(a_n)b_\g)=\varphi(i_{j_1}(a_1)\cdots i_{j_n}(a_n))\om(b)\,,$$
and we are done because by our inductive hypothesis we have
$$\varphi(i_{j_1}(a_1)\cdots i_{j_n}(a_n))\om(b)=\om(a_1)\cdots\om(a_n)\om(b)\, .$$
The last thing to do is to notice that the set of product states is closed in the $*$-weak topology. 
\end{proof}
Before restating the above result in terms of the corresponding processes with distribution that factorizes through an infinite twisted tensor product, we need to observe that $\cs_v(\ga)\ni\om\mapsto \times\om\in{\rm Extr}(\cs^{\bj_\bn}(\bigotimes_v^\bn \ga)) $ is a homeomorphism between compact Hausdorff spaces.

\begin{cor}\label{definettiprocesses}
Let $v:\widehat{G}\times\widehat{G}\rightarrow\bt$ be a bicharacter such that either $v(\chi, \eta)=1$
or $v_S(\chi, \eta)\neq 1$ for $\chi, \eta$ in $\Delta_v$.
For any spreadable process whose distribution $\varphi$ factorizes through $\bigotimes_v^\bn\ga$, there exists a unique Borel probability measure $\mu$ on $\cs_v(\ga)$ such that
$$\varphi(i_{j_1}(a_1)\cdots i_{j_n}(a_n))=\int_{\cs_v(\ga)} \om(a_1)\cdots\om(a_n)\, {\rm d}\mu(\om)\,,$$
for all $n$ in $\bn$, $j_1, \ldots, j_n\in\bn$, $a_1, \ldots, a_n$ in $\ga$, where $i_j: \ga\rightarrow\ast_\bn\ga$ is the $j$-th canonical embedding of $\ga$ into $\ast_\bn\ga$.
\end{cor}

\subsection{Examples}
If $G=\bt$, then $\widehat{G}=\bz$, and  all bicharacters $v:\bz\times\bz\rightarrow\bt$ are of the form
$v_\alpha(k, l)= e^{2\pi i \a kl}$, $k, l$ in $\bz$, for some real $\alpha$. In particular, all bicharacters of
$\bz$ are symmetric. If $\frac{\alpha}{2\pi}$ is irrational, $\Delta_{v_\alpha}=\{0\}$. If $\frac{\alpha}{2\pi}$ is rational,
$\Delta_{v_\alpha}$ is a subgroup of $\bz$ and $v\upharpoonright_{\Delta_{v_\alpha}\times \Delta_{v_\alpha}}=1$, see \cite[Lemma 3.3]{CDGR}. Therefore, the sufficient conditions in Theorem \ref{maindefinetti} for $H$-abelianness to hold are satisfied in both cases. \\
For example, we can take $\ga= C(\bt)$ with the $\bt$-grading induced by rotations, {\it i.e.}
$\gamma_z (f)\, (w)=f(zw)$, $z, w$ in $\bt$. The corresponding infinite twisted tensor product is
(isomorphic with) the infinite non-commutative torus, whose spreadable states have already been addressed in
\cite{CDGR}. What turned out from the analysis carried out in that paper, see \cite[Theorem 3.8]{CDGR}, but also follows from the general picture given in the present, Theorem \ref{maindefinetti}, is that
extreme spreadable states are infinite product of a single state (measure) $\om$ on $C(\bt)$ which has to be invariant for the action of the grading restricted to $\Delta_{v_\alpha}^\perp$. In the irrational case, 
$\Delta_{v_\alpha}^\perp=\bt$, and thus the trace is the only (extreme) spreadable state.
In the rational case $\Delta_{v_\alpha}^\perp=\{z\in\bt: z^{n_0}=1\}$ is a finite group ($n_0$ is an integer associated with $\alpha$, see \cite[Lemma 3.3]{CDGR}), and many more extreme spreadable states exist.
\begin{rem}\label{cyclicgroups}
Gradings induced by finite cyclic groups, say $\bz_n$, are a particular case of the above class of examples, in that $\bz_n$ can be seen as a quotient
of $\bz$.
\end{rem}
We now go back to the general framework  focussing 
on those bicharacters $v:\widehat{G}\times\widehat{G}\rightarrow\bt$ that  are trivial on $\Delta_v\times\Delta_v$. The resulting
 de Finetti theorem is basically classical in a sense that we  specify in the corollary below.
\begin{cor}
If $v: \widehat{G}\times\widehat{G}\rightarrow\bt$ is a bicharacter such that
$v\upharpoonright_{\Delta_v\times\Delta_v}=1$, then the map
$$\cs^{\bj_\bn}(\otimes(\ga^{{\Delta_v}^\perp}))\ni\om\mapsto \om\circ E\in \cs^{\bj_\bn}(\otimes_v^\bn\ga)$$
is an affine homeomorphism between Bauer simplices, where $E: \otimes_v^\bn\ga\rightarrow\otimes(\ga^{{\Delta_v}^\perp})$ is the product conditional expectation.
\end{cor}
\begin{proof}
We start by noting that under the hypothesis $v\upharpoonright_{\Delta_v\times\Delta_v}=1$, $\Delta_v$ is a subgroup of $\widehat{G}$. In particular,  $\overline{{\rm span}}\{a\in \ga: \partial a\in\Delta_v\}$ is a
$C^*$-subalgebra and coincides with $\ga^{\Delta_v^\perp}$.
Over the course of the proof of Theorem \ref{maindefinetti}, we showed that the map in the statement is a homeomorphism
between the extreme sets of the two simplices. The full conclusion the follows from the fact that our simplices are 
Bauer simplices.
\end{proof}

In particular, if the sample  $C^*$-algebra $\ga$ is commutative, then $\ga^{\Delta_v^\perp}$ is commutative as well, and the Bauer simplex of spreadable states on the whole chain is thus affinely homeomorphic
with that of  the distributions of spreadable sequences of random variables  taking their values on the spectrum of $\ga^{\Delta_v^\perp}$, and to this convex set the classical theory applies.\\
Notably, this situation includes the analysis of spreadable (or equivalently exchangeable) states on the CAR algebra.\\
If $\ga$ is no longer commutative, $\ga^{\Delta_v^\perp}$ will be in general non-commutative. However,
exchangeable states on an infinite tensor product have been thoroughly classified in the classical work of St\o rmer
\cite{Sto}.\\

Our techniques also apply to different settings, for example to gradings induced by finite groups
of the form $\bz_{n_1}\times....\times\bz_{n_k}$, with $n_1, \ldots, n_k$ odd, in light of the following.

\begin{prop}\label{finprod}
Any symmetric bicharacter $v$ on the product $\bz_{n_1}\times....\times\bz_{n_k}$, where $n_1, \ldots, n_k$ are all odd, satisfies the sufficient conditions of $H$-abelianness in Theorem \ref{suffcondabelianness}, namely
$v_S(\chi, \eta)\neq 1$ for $\chi, \eta$ in $\Delta_v$.
\end{prop}

\begin{proof}
We start by observing that, given a general symmetric bicharacter $v$, for $v_S$ not to be $1$ on $\Delta_v\times\Delta_v$ is enough to verify
that $v(\chi, \eta)$ is never $-1$, for $\chi, \eta$ in $\Delta_v$.\\
There is yet another general fact to recall. If $g$ in $G$ is an element of finite order $n$ and
$\chi\in\widehat{G}$ is a character, then $\chi(g)\in\bt$ is a $n$-th root of unity.
We now have all is needed to conclude: every element of $G=\bz_{n_1}\times....\times\bz_{n_k}$
has odd order if all $n_j$'s are odd; in particular, it follows that $v(\chi, \eta)$ is a root of unity of odd order for all
$\chi, \eta$ in $\bz_{n_1}\times....\times\bz_{n_k}$, hence $v(\chi, \eta)$ can never equal $-1$.
\end{proof}
\begin{rem}
The argument employed in the above proposition also applies to infinite direct products $\times_{i\in\bn} \bz_{n_i}$, as long as all
$n_i$'s are odd.
\end{rem}

The cases covered in Proposition \ref{finprod} include examples where $v$ is not constantly  $1$ on $\Delta_v\times \Delta_v$. As a result, the simplex of spreadable states on the corresponding chain cannot be obtained out of a single usual tensor product, unlike all known examples as far as we know. Below we discuss a concrete example of this sort.
 
\begin{example}\label{definettifoliated}
Consider $G=\bz_3\times\bz_3$, and $v: (\bz_3\times\bz_3)^2\rightarrow\bt$ the bicharacter given by
$v(\chi, \eta) := e^{\frac{2\pi i}{3}\langle A\chi, \eta\rangle}$, for $\chi, \eta$ in $\bz_3 \bz_3$, where
$A$ is the $2$ by $2$ matrix $A := \left( \begin{array}{cc} 0 &1\\ 1 &1\\ \end{array} \right)$.
By direct computation, one can verify that $\Delta_v=\{(0, 0), (1, 1), (2, 2), (1, 0), (2, 0)\}$.
Note that $v$ is not $1$ on $\Delta_v\times\Delta_v$, for $v((1, 0), (1, 1))=e^{\frac{2\pi i}{3}}$.
We also point out that $\Delta_v$ contains two maximal subgroups $G_1, G_2$ (maximal w.r.t. the property that $v\upharpoonright_{G_i\times G_i}=1$ for $i=1, 2$), and these are $G_1=\{(0,0), (1, 1), (2, 2)\}$ and $G_2=\{(0, 0), (1, 0), (2, 0)\}$.\\
Consider an infinite chain based on the sample $\ga$ graded by $\bz_3\times\bz_3$. For $i=1, 2$, let
$\gb_i$ be the $C^*$-subalgebra of the chain given by $\otimes_\bn \ga^{G_i}$. 
The simplex of spreadable states on this chain is then the Bauer simplex whose extreme set is homeomorphic with   $\cs^{G_1}(\ga)\cup \cs^{G_2}(\ga)\subset \cs(\ga)$. In particular, one can take
$\ga= C(\bz_3\times\bz_3)$ acted upon by $\bz_3\times\bz_3$ by rotation. With such an $\ga$, $\cs^{G_1}(\ga)\cup \cs^{G_2}(\ga)$ is a finite set and, more importantly, $\cs^{G_1}(\ga)$ and $ \cs^{G_2}(\ga)$ are different sets, meaning there is not a single tensor subalgebra all spreadable states arise from by composing with the canonical conditional expectation onto it.
\end{example}
\subsection{Braided Parafermions}\label{parafermion}
Our techniques also cover the models of braided parafermions considered in \cite{BJLW}, which we next recall. For each natural $d\geq 2$, set $q:=e^{\frac{2\pi i}{d}}$.  The sample algebra $\ga$ of the models is the full matrix algebra of $d\times d$ matrices $
\mathbb{M}_d(\mathbb{C})$, thought of as the universal $C^*$-algebra generated by two unitaries $c_1, c_2$ such that
$c_j^d=1$, $j=1, 2$, and $c_1c_2=qc_2c_1$. The sample is conceived of as a $\bz_d$-graded
$C^*$-algebra, with the grading being induced by the automorphism (of order $d$) $\theta$ uniquely determined by
$\theta(c_j):=q c_j$, $j=1, 2$. To obtain the chain, we also need to specify what is bicharacter to consider.
This is $v: \bz_d\times\bz_d\rightarrow\bt$ given by $v(k, l)= q^{kl}$, $k, l$ in $\bz_d$.
As is straightforward to check, with these data, $\otimes_v^\bn\ga$ is just  the $C^*$-algebra $PF_{\infty}$ considered in \cite{BJLW}, that is the universal $C^*$-algebra generated by a sequence $\{c_j; j\in\bn\}$ of unitaries such that $c_j^d=1$, for all $j$ in $\bn$, and $c_jc_k=qc_kc_j$ for all $j<k$.\\
The authors then consider a suitable action of the braid group on
$PF_{\infty}$ by inner automorphisms, \cite[Formulas (10), (11), (12)]{BJLW}, and study the corresponding invariant states. Furthermore, the action they define is compatible with Gohm and Kostler's definition of braidability.
What the authors find, among other things, is that the extreme braidable states are infinite product states of a single state $\om$ on $\mathbb{M}_d({\bc})$.
 However, possibly not all such states are braidable w.r.t. the inner action, and it is actually an open problem to determine which of these states are braidable. 
There are two cases depending on the prime factorization of $d$. If $d=p_1p_2\cdots p_n$, with $p_i$ being prime numbers, then $\om$ is forced to be $\bz_d$-invariant for its infinite product to exist, whereas
if the factorization of $d$ contains a power of a prime with exponent at least $2$, then more sample
states give product states, {\it cf.} \cite[Theorem 17, Theorem 21]{BJLW}.
In our language, the two cases naturally correspond to whether the isotropy subgroup $\Delta_v$ is trivial or not, see \cite[Lemma 3.3]{CDGR} and Remark \ref{cyclicgroups}.\\
Since $PF_\infty$ is an infinite twisted tensor product, we are in a position to obtain a thorough characterization of all spreadable states on  $PF_\infty$.
Precisely, the Bauer simplex
of all spreadable states can be determined with our techniques, namely Theorem \ref{maindefinetti} and
Remark \ref{cyclicgroups}. Furthermore, the set of all braidable states on the parafermion algebra is a Bauer simplex as well as being a face of the simplex of the spreadable states, as we
show below.

\begin{cor}\label{braidspread}
The set of all braidable states of  $PF_\infty$ w.r.t. the inner action is a Bauer simplex which is also a face of the Bauer simplex of all spreadable states on  $PF_\infty$.
\end{cor}

\begin{proof}
Since the proof of both Lemma \ref{commrulproj} and Theorem \ref{suffcondabelianness} applies to the action of the braid group as well, we have that
the convex set of braidable states is again a Choquet simplex. Moreover, it is a Bauer simplex in that its boundary is closed, being the intersection of the set of all product states with the set of all states invariant under the action of the braid group.\\
In order to prove that braidable states make up a face of the larger simplex, it is enough to reason as in the proof of Proposition 3.5 in \cite {CDRCAR},  and observe that extreme states of the smaller set are extreme in the larger as extreme braidable states are product states by the analysis in \cite{BJLW}.
\end{proof}

\section{Stationary states on infinite twisted tensor products}
This last section is concerned with stationary states on infinite twisted products. 
We start with a simple observation which needs no proof.
\begin{prop}
The restriction map $T: \cs^\tau\big(\bigotimes_v^{\bz}\ga\big)\rightarrow \cs^{\tau}(\bigotimes_v^\bn\ga)$
$$T(\varphi)=\varphi\lceil_{\bigotimes_v^\bn\ga}$$
establishes an affine homeomorphism of compact convex sets.
\end{prop}

%
The support of an extreme shift-invariant state is subject to the following compatibility condition with the bicharacter, which is reminiscent of
the necessary and sufficient condition guaranteeing the existence of the product state of a given state.

\begin{prop}\label{tauextprop}
Suppose $(\bigotimes_v^{\bz}\ga, \tau)$ is
$\bz$-abelian.
If $\varphi$ is an extreme shift-invariant state on $\bigotimes_v^\bz\ga$, then  $v(\chi, \eta)=1$ for $\chi, \eta$ in
${\rm supp }_G\,\varphi$.
\end{prop}

\begin{proof}
Let $\chi, \eta$ in $\widehat{G}$ two characters belonging to the spectral support of $\varphi$.
By definition, there exist homogeneous $a, b$ in $\bigotimes_v^\bz\ga$ such that
$\partial a=\chi$, $\partial b=\eta$ and $\varphi(a), \varphi(b)\neq 0$.
We claim that $ E_\varphi\pi_\varphi(a)E_\varphi\pi_\varphi(b)E_\varphi$ is different from zero. Indeed, we have
\begin{align*}
&\langle E_\varphi\pi_\varphi(a)E_\varphi\pi_\varphi(b)E_\varphi \xi_\varphi, \xi_\varphi \rangle=
\langle \pi_\varphi(a )E_\varphi\pi_\varphi(b) \xi_\varphi, \xi_\varphi \rangle=\\
&\langle \pi_\varphi(a) \langle \pi_\varphi(b)\xi_\varphi, \xi_\varphi \rangle\xi_\varphi, \xi_\varphi  \rangle=
\varphi(a)\varphi(b)\neq 0\,
\end {align*}
since $E_\varphi$ is the projection onto $\bc\xi_\varphi$ by extremality of $\varphi$ and $\bz$-abelianness.
Now this implies $v(\chi, \eta)=1$, for otherwise $E_\varphi\pi_\varphi(a)E_\varphi$ and $E_\varphi\pi_\varphi(b)E_\varphi$
would not commute.
\end{proof}
The main goal we want to accomplish now is to say when stationary states make up a convex compact set affinely homeomorphic with the Poulsen simplex, that is the unique metrizable Choquet simplex whose extreme set is dense, see \cite{P, Lind}.
To this end, we need the following result, which gives information on the abundance of product states.

\begin{lem}\label{existence}
For any $\chi$ in $\Delta_v$ such that $V_\chi\neq 0$, there exists a product
 state $\varphi$  such that ${\rm supp}_G\,\varphi$ contains $\chi$ (and ${\rm supp}_G\, \varphi\subseteq \{\chi^n: n\in\bz\}\subset \widehat{G}$).
\end{lem}

\begin{proof}
Such a state $\varphi$ can be obtained by applying Theorem \ref{existenceprod},  once we have exhibited
a state $\om$ on $\ga$ such that $\chi\in {\rm supp}_G\, \om$ and
$v\upharpoonright_{{\rm supp}_G\, \om\times{\rm supp}_G\, \om}=1$.
To this aim, let $H=\{\chi^n: n\in\bz\}\subset \widehat{G}$
be the subgroup generated by $\chi$ and let $\ga_H:=\overline{{\rm span}}\{a: \partial a\in H\}$.
Let $F: \ga\rightarrow \ga_H$ be the canonical conditional expectation of $\ga$ onto $\ga_H$, that is the one obtained
by averaging over $H^\perp\subset G$ (w.r.t. the Haar measure of $H^\perp$) the action of the grading.
The sought state $\om$ can be taken of the form $\om=\om_0\circ F$, where
$\om_0$ is a state on $\ga_H$ such that $\om_0(a)\neq 0$ for some
non-zero $a$ in $\ga_H$ with $\partial a=\chi$.  
\end{proof}

To ease the notation, we set $C_\tau:=\cs^\tau\big(\bigotimes_v^{\bn}\ga\big)$.
\begin{thm}\label{sufficientPoulsen}
If the bicharacter $v: \widehat{G}\times\widehat{G}\rightarrow\bt$ is such that $v\upharpoonright_{\Delta_v\times \Delta_v}=1$, then the compact convex set $C_\tau$ is (affinely homeomorphic with) the Poulsen simplex.\\
Furthermore, if the dynamical system $(\ga, G, \gamma)$ is ergodic and the action is faithful,
then  $C_\tau$ is (affinely homeomorphic with) the Poulsen simplex if and only
if $v\upharpoonright_{\Delta_v\times \Delta_v}=1$.
\end{thm}

\begin{proof}
We first observe that the condition on $v$ guarantees that $\Delta_v$ is a subgroup of $\widehat{G}$. Indeed,
for $\chi, \eta\in \Delta_v$ ({\it i.e.} $v(\chi, \chi)=v(\eta, \eta)=1$), we have $v(\chi\eta, \chi\eta)=v(\chi, \chi) v(\eta, \eta)v(\chi,\eta )v(\eta, \chi)=1$.
This implies that $\ga_{\Delta_v} := \overline{{\rm span}}\{a\in\ga: \partial a\in \Delta_v\}$ is
a $C^*$-subalgebra of the sample algebra $\ga$.\\
Furthermore, the condition on the bicharacter also allows us to apply Theorem \ref{suffcondabelianness}, which means $C_\tau$ is a (metrizable) Choquet simplex.
All we need to do is show that ${\rm Extr}(C_\tau)$ is dense in $C_\tau$.
This can be done in the exact same way as in \cite[Theorem 4.3]{CDGR} by exploiting the fact that
$\ga_v^\bz$ contains a copy of the maximal infinite tensor product of $\ga_{\Delta_v}$ with itself.\\
In light of what we saw above, we need only prove that, under the assumption $(\ga, G, \gamma)$ is ergodic and the action $\gamma$ is faithful, if  $v\upharpoonright_{\Delta_v\times \Delta_v}=1$ fails to hold, then the extreme set of $C_\tau$ is not dense in $C_\tau$.
We first note that we do not harm the generality if we assume $\bz$-abelianness. Indeed, if $\bz$-abelianness does not hold, then $C_\tau$ is not even a Choquet simplex by Lemma \ref{summary}.\\
Let $\chi, \eta$ in $\Delta_v$ such that $v(\chi, \eta)=1$.
Since $V_\chi$ and $V_\eta$ are both non-zero thanks to Proposition \ref{fullspectrum},  Lemma \ref{existence} yields two distinct (extreme) shift-invariant states
$\varphi_1, \varphi_2$ with $\chi \in {\rm supp}\,{\varphi_1}$ and $\eta \in {\rm supp}\,{\varphi_2}$.
We now consider any non-trivial convex combination $\om := t\varphi_1+(1-t)\varphi_2$ ($0<t<1$).
From Proposition \ref{tauextprop}, it follows that $\om$ cannot be approximated by a sequence of extreme states.
\end{proof}

\appendix
\section{}

The following result is a (known) useful criterion to decide if an automorphism of a given $C^*$-algebra
passes to a quotient.
\begin{lem}\label{quotientlemma}
Let $\ga$ be a (unital) $C^*$-algebra and $I\subset\ga$ a (proper) closed two-sided ideal,
and $p:\ga\rightarrow\ga/I$ be the canonical projection onto the quotient $C^*$-algebra.
Let $\Phi$ be an automorphism of $\ga$. If there exists a faithful state $\varphi$ on 
$\ga/I$ such that $\varphi\circ p$ is a $\Phi$-invariant state of $\ga$, then $\Phi(I)\subseteq I$.
\end{lem}

\begin{proof}
The GNS representation of $\om\doteq \varphi\circ p$ is (unitarily equivalent to) $\pi_\varphi\circ p$.
In particular, since $\varphi$ is faithful, $I={\rm Ker}\,\pi_\om$. Therefore, we need to make sure $\Phi({\rm Ker}\,\pi_\om)\subseteq{\rm Ker}\,\pi_\om$.
Let $a$ in $\ga$ be such that $\pi_\om (a)=0$. Then $\pi_\om(a)\pi_\om(b)\xi_\om=0$, for every $b$ in $\ga$, hence $\om(b^*a^*ab)=0$. Because $\om$ is $\Phi$-invariant, we have
$\om(\Phi(b^*a^*ab))=0$, that is $\pi_\om(\Phi(a))\pi_\om(\Phi(b))\xi_\om=0$, hence
$\pi_\om(\Phi(a))=0$ by density of $\{\pi_\om(\Phi(b))\xi_\om: b\in\ga\}\subset\ch_\om$.
\end{proof}

\section*{Acknowledgments}
\noindent
V.A. is partially supported by Sapienza Universit\`a di Roma (Progetto di Ateneo Dipartimentale 2024 "New research trends in Mathematics at Castelnuovo").
The authors are all members of INDAM-GNAMPA.
S.D.V. and S.R. acknowledge the support of INdAM-GNAMPA Project 2024 “Probabilita’ quantistica e applicazioni”, 
 the support of the Italian INDAM-GNAMPA Project Code CUP\_E55F22333270001,
 Centro Nazionale CN00000013 CUP H93C22000450007,
the Italian PNRR MUR project PE0000023-NQSTI,
by Italian PNRR Partenariato Esteso PE4.


\end{document}